\documentclass{article}
\usepackage{amsmath}
\usepackage{color}
\usepackage{pgfplots}
\pgfplotsset{compat=1.17}

\usepackage{amsthm}
\usepackage{dsfont}
\newtheorem{theorem}{Theorem}[section]
\newtheorem{lemma}[theorem]{Lemma}

\newtheorem{remark}[theorem]{Remark}

\newtheorem{hypothesis}[theorem]{Hypothesis} 

\usepackage{amsthm,amsmath,amssymb}
\usepackage{amsthm}
\usepackage{mathrsfs}
\usepackage{enumitem}
\usepackage{graphicx}
\usepackage{subcaption} 
\makeatother
\usepackage{mathrsfs}

\numberwithin{equation}{section}
\providecommand{\keywords}[1]
{
  \small	
  \textbf{\textit{Keywords: }} #1
}

\usepackage{cite}
\usepackage[backref]{hyperref} 

\title{Acceleration or finite speed propagation in weakly monostable reaction-diffusion equations}
\author{Emeric Bouin \footnote{CEREMADE - Universit\'e Paris-Dauphine, PSL Research University, UMR CNRS 7534, Place du Mar\'echal de Lattre de Tassigny, 75775 Paris Cedex 16, France. E-mail: \texttt{bouin@ceremade.dauphine.fr}}
	\and J\'er\^ome Coville \footnote{UR 546 Biostatistique et Processus Spatiaux, INRA, Domaine St Paul Site Agroparc, F-84000 Avignon, France. E-mail: \texttt{jerome.coville@inrae.fr}}
	\and Xi Zhang \footnote{School of Mathematics and Statistics, Central South University, Changsha, Hunan 410083, P. R. China. E-mail:  \texttt{xizhangmath@gmail.com}}}

\begin{document}
	\maketitle
	\begin{abstract}
This paper focuses on propagation phenomena in reaction-diffusion
equations with a weakly monostable nonlinearity. The reaction term can be seen as an intermediate between the classical logistic one (or Fisher-KPP)  and the standard weak Allee effect one.
We investigate the effect of the decay rate of the initial data on the propagation rate.  When the right tail of the initial data is sub-exponential, finite speed propagation and acceleration may happen and we derive the exact separation between the two situations.  When the initial data is sub-exponentially unbounded, acceleration unconditionally occurs. Estimates for the locations of the level sets are expressed in terms of the decay of the initial data. In addition, sharp exponents of acceleration for initial data with sub-exponential and algebraic tails are given. Numerical simulations are presented to illustrate the above findings.

 \end{abstract}

	\section{Introduction}\label{s1}
	In this paper, we study rates of invasion in the following one-dimensional reaction-diffusion equations
	\begin{equation}\label{oeq1}
 \left\{
 \begin{aligned}
	 u_t(t,x)&=u_{xx}(t,x)+f(u(t,x)),&t>0, x\in\mathbb{R},\\
  u(0,x)&= u_0(x)\ge 0, & x\in\mathbb{R}.
  \end{aligned}
  \right.
	\end{equation}

		\begin{hypothesis}\label{fu}
		The non-linearity $f\in C^1([0,1], \mathbb{R})$ is of the weakly monostable type, in the sense that
		\begin{equation*}
			f(0)=f(1)=0,\quad  f(s)>0  \text{  for any }s\in(0,1),\quad f'(1)<0,
		\end{equation*}
 and there exists $s_0\in (0,1)$, $K\ge0$, $\alpha>0$ and $r>0$ such that 
		\begin{equation}\label{fu1}
			f(s)\le r\frac{s}{(1+|\ln s|)^\alpha} \quad \text{for  all } s\in(0,1),
		\end{equation}
  and 
		\begin{equation}\label{fu2}
		f(s)\ge r\frac{s}{(1+|\ln s|)^\alpha}(1-Ks) \quad \text{for  all } s\in(0,s_0].
		\end{equation}
	\end{hypothesis}
After Kolmogorov, Petrovskii and Piskunov \cite{kolmogorov1937etude}, and Fisher \cite{fisher1937wave},  the classical monostable equation is equation \eqref{oeq1} with Fisher-KPP type nonlinearity, that is,
\begin{equation}\label{kppass}
    f(0)=f(1)=0 \text{ and } 0<f(s)\le f'(0)s \text{ for all }s\in(0,1).
\end{equation}
In population dynamics, this type of non-linearity is commonly used to model the situation where growth per capita is maximal at low densities.  The decay rate of the initial data at infinity is crucially important for the propagation problem. For the Fisher-KPP equation with front-like initial data,  initial data $u_0$ with exponentially bounded decay, that is, 
\begin{equation}\label{ebin}
    \lim_{x\to +\infty} u_0(x)e^{\varepsilon x}<\infty \quad \text{for some }\varepsilon,
\end{equation}
lead to finite propagation speed \cite{mallordy1995parabolic,bramson1983convergence}.
On the other hand, for an exponentially unbounded initial data, meaning that condition \eqref{ebin} is not met, or
\begin{equation}\label{eubin}
 \lim_{x\to +\infty} u_0(x)e^{\varepsilon x}> +\infty \quad \text{for any }\varepsilon,
\end{equation}
Hamel and Roques \cite{2009Fast} have presented evidence of acceleration of the solution to the Fisher-KPP equation. They also provided an expression of the locations of level sets based on the decay of the initial data. We refer to references \cite{garnier2011accelerating,cabre2013influence,alfaro2017propagation,henderson2016propagation,lau1985nonlinear,du2023exact,hamel2010uniqueness,berestycki2008asymptotic,berestycki2005speed,berestycki2010speed} for  the further results about propagation in KPP equations.

\par
When an Allee effect occurs, meaning that the per capita growth is no longer maximal at low densities, the KPP assumption \eqref{kppass} becomes unrealistic. Hence, incorporating the Allee effect into models becomes necessary. An acceleration phenomenon may take place in the degenerate situation $f'(0)=0$. Indeed, when the initial data is front-like and the nonlinearity $f(s)\sim r s^{\alpha+1}$ with $\alpha>0$ as $s\to 0^+$, Alfaro \cite{alfaro2017slowing} has studied the balance between the decay rate of the initial data at infinity and the weak Allee effect and found that for exponentially unbounded tails but lighter than algebraic acceleration does not occur in the presence of the Allee effect, which is in contrast with the KPP equation. Similarly to the KPP situation, the initial data with exponentially bounded decay lead to a finite propagation speed \cite{kay2001comparison,roquejoffre1997eventual}. On the other hand, algebraic decay leads to acceleration despite the Allee effect and the position of the level sets of $u(t,\cdot)$ as $t\to \infty$ propagates polynomially fast \cite{alfaro2017slowing,leach2002evolution}. We refer to  references \cite{bates1997traveling,chen1997existence,2013Traveling,gui2015travelingalleecahn} for other kinds of Allee effect. 
\par
It is worth mentioning that these results about propagation phenomena in degenerate monostable equations are based on the assumption $f(s)\sim r s^{\alpha+1}$ with some $\alpha>0$ and $r>0$ as $s\to 0^+$.  This assumption is used to quantify the degeneracy. In this paper, we also take into account that the growth per capita is small at small densities, but we quantify the degeneracy by a weakly monostable type nonlinearity $f$ satisfying $f(s)\sim r\frac{s}{|\ln s|^\alpha}$ with $\alpha>0$ and $r>0$ as $s\to 0^+$, like $f(s)=r \frac{s}{(1+|\ln s|)^\alpha}(1-s)$ for $s\in(0,1)$. Notice that such nonlinearity is between the KPP type and the Allee effect type near the right side of zero point, see Figure \ref{figcom}. Thus, this type of nonlinear term fill an existing gap  between two classical nonlinearities.

\begin{figure}
    \centering
\begin{tikzpicture} 
\begin{axis}[
xlabel = $s$,
ylabel = $\frac{f(s)}{s}$,
xmin=-0.01,
xmax=0.1,
ymin=0,
ymax=1.2,
xtick = {0.00},
xticklabels = {0},
ytick =\empty,
]
 \addplot[
 color=blue,
 samples = 100, 
 domain= 1e-10:0.1,
 line width =1.5,
 dotted]{x/x};
\addlegendentry{$r$};
    \addplot[
    color = red,
    line width = 3.0,
    samples = 600,
    domain = 1e-100:0.1,
    ]{1/(-log10(x))};
    \addlegendentry{$\tfrac{r}{|\ln s|^\alpha}$};  
    \addplot[color = black,samples = 100, domain= 0:0.1,line width =1.5, dashed]{x^0.5};
        \addlegendentry{$r s^\alpha$};
    \end{axis}
\end{tikzpicture}
\begin{tikzpicture} 
\begin{axis}[
xlabel = $s$,
ylabel = $f(s)$,
xmin=-0.1,
ymin=0,
ymax=0.4,
xtick = {0,1},
xticklabels = {0,1},
ytick =\empty,
]
 \addplot[
 color=blue,
 samples = 100, 
 domain= 1e-10:1,
 line width =1.5,
 dotted]{x*(1-x)};
\addlegendentry{$rs(1-s)$};
    \addplot[
    color = red,
    line width = 3.0,
    samples = 600,
    domain = 1e-100:1,
    ]{x*(1-x)/(1-log10(x))};
    \addlegendentry{$r\frac{s(1-s)}{(1+|\ln s|)^\alpha}$};  
    \addplot[color = black,samples = 100, domain= 0:1, line width =1.5, dashed]{x^2*(1-x)};
        \addlegendentry{$r s^{\alpha+1}(1-s)$};
    \end{axis}
\end{tikzpicture}
  \caption{Comparison of the size of three kinds of nonlinearities near zero, where the parameter $r$ and $\alpha$ are positive.}
 \label{figcom}
\end{figure}
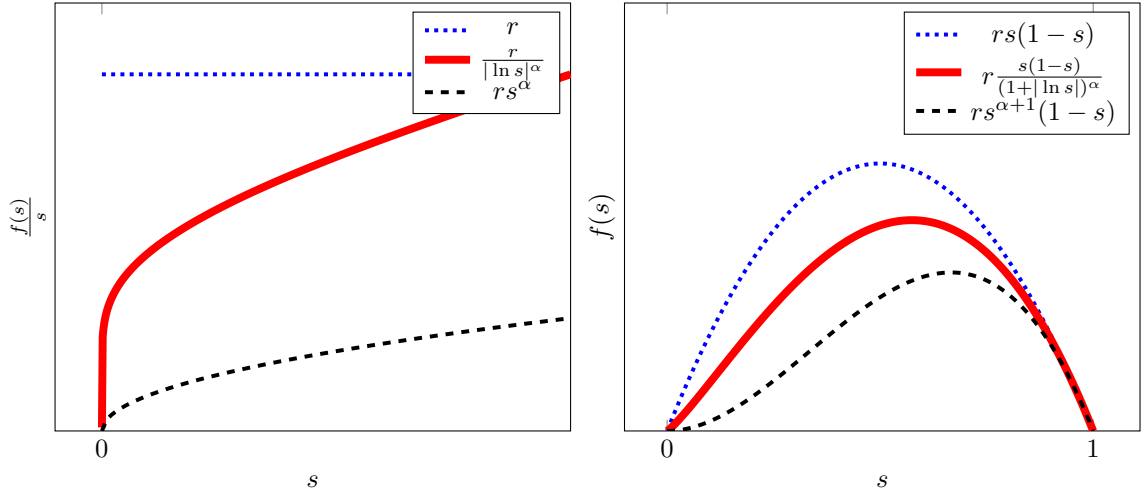
\par	To describe the propagation speed, we introduce three notations. For any $\lambda\in(0,1)$, the (upper) level set of $u(t,x)$ is defined by
\begin{equation*}
	E_\lambda(t):=\{x\in\mathbb{R}:u(t,x)\ge \lambda\}.
\end{equation*}
Let $x_\lambda(t)$ be the largest element of level set of $u(t,x)$  defined by
\begin{equation*}
	x_\lambda(t):=\sup E_\lambda(t).
\end{equation*}
For any subset $\Lambda\subset (0,1]$, we set
\begin{equation*}
    u_0^{-1}\{\Lambda\}:=\{x\in\mathbb{R}:u_0(x)\in \Lambda\}
\end{equation*}
the inverse image of $\Lambda$ by $u_0$.

 \begin{hypothesis}\label{ic}
The initial data $u_0:\mathbb{R}\to [0,1]$ is uniformly continuous and asymptotically front-like, in the sense that
		\begin{equation*}
			u_0>0 \ \text{in} \ \mathbb{R}, \quad \liminf_{x\to-\infty}u_0>0,\quad \lim_{x\to +\infty}u_0=0.
		\end{equation*}
  \end{hypothesis}
In this paper, we always denote by $u(t,x)$ the solution to \eqref{oeq1} with initial data $u_0$. We mainly consider the following types of initial data:
\begin{itemize}
    \item Sub-exponentially bounded for large $x$, that is, there exist $x_0>0$ such that, for any $x>x_0$,
 \begin{equation*}
     u_0(x)\lesssim e^{-\mu x^\beta},
 \end{equation*}
 with $\beta\in(0,1)$ and $\mu>0$.
 \item Sub-exponential decay for large $x$, that is, there exist $x_0>0$ such that, for any $x>x_0$, 
 \begin{equation*}
     u_0(x)\asymp e^{-\mu x^\beta},
 \end{equation*}
 with $\beta\in(0,1)$ and $\mu>0$ \footnote{The notation $a\asymp b$ means that there exists a constant $C$ such that $C b\le a \le C^{-1} b$.}.
 \item Algebraic decay for large $x$, that is, there exists $x_0>1$ such that, for any $x>x_0$,
 \begin{equation*}
     u_0(x)\asymp \frac{1}{x^\beta},
 \end{equation*}
 with $\beta>0$.
 \item Initial data $u_0$ that decay as a negative power of $\ln x$ for large $x$, that is, there exists $x_0>e$ such that, for any $x>x_0$,
 \begin{equation*}
     u_0(x)\asymp (\ln x)^{-\beta},
 \end{equation*}
 with $\beta>0$.
 \end{itemize}

\par  Our first result shows that for sub-exponentially bounded initial data, acceleration does not happen.
	\begin{theorem}\label{lin}
			Let $\alpha>0$ and $\beta>0$ be such that 
   \begin{equation*}
       \beta\ge\frac{1}{\alpha+1}.
   \end{equation*}
   Assume that the non-linearity $f$ and the initial data $u_0$ satisfy Hypotheses \ref{fu} and \ref{ic}, respectively. Assume that there exist $x_0>0$ and $\mu>0$ such that
		\begin{equation}\label{x0}
			 u_0(x)\lesssim  e^{-\mu x^\beta} \quad \text{for any } x\ge x_0.
		\end{equation}      
Then, for any $\lambda\in(0,1)$, there exist some positive constants $c$ and a time $T_\lambda$ such that 
		\begin{equation}\label{linineq}
			\Gamma <\frac{x_\lambda(t)}{t}<c  \quad \text{for any }t>T_\lambda.
		\end{equation}
\end{theorem}

 Now, we turn to cases where it is assumed that the initial data $u_0$ decay more slowly than $e^{-\varepsilon x^{\frac{1}{\alpha+1}}}$ as $x\to +\infty$ for any $\varepsilon>0$, that is,
 \begin{equation}\label{sed}
    \forall \varepsilon>0,\quad \exists x_\varepsilon\in \mathbb{R},\quad u_0(x)\ge e^{-\varepsilon x^{\frac{1}{\alpha+1}}} \text{in }[x_\varepsilon,+\infty). 
 \end{equation}
Let us denote 
\begin{equation}\label{lvar}
    \varphi_0(x) := -\ln u_0(x)\ge 0.
\end{equation}
Notice that if $u_0$ is $C^2$, then we can get
\begin{equation*}
    \varphi_0'(x)=-\frac{u_0'}{u_0}(x) \quad\text{and}\quad \varphi_0''(x)=-\Big(\frac{u_0'}{u_0}\Big)'(x).
\end{equation*}
Observe that if we assume that $\varphi_0'(x)=o(\varphi_0^{-\alpha}(x))$ as $x\to +\infty$, then condition \eqref{sed} is fulfilled. 
\par For such initial data, we have the following result. 
	\begin{lemma}\label{gen}
	 Assume that the non-linearity $f$ and the initial data $u_0$ satisfy Hypotheses \ref{fu} and \ref{ic}, respectively. Assume that $u_0$ is of class $C^2$ and  non-increasing on $[\xi_0,+\infty)$ for some $\xi_0>0$, and
		\begin{equation}\label{genoo}
			\varphi_0'(x)=o(\varphi_0^{-\alpha}(x)) \text{  and  }\varphi_0''(x)=o(\varphi_0'(x))  \quad \text{as }x\to +\infty.
		\end{equation}
  Then, for any fixed $\lambda\in(0,1)$ and small $\varepsilon>0$, there is a time $T_{\lambda,\varepsilon}$ such that
		\begin{equation*}
			E_\lambda(t)\subset u^{-1}_0\left\{\left[e^{-[(r+\varepsilon)(\alpha+1)t]^{\frac{1}{\alpha+1}}},e^{-[(r-\varepsilon)(\alpha+1)t]^{\frac{1}{\alpha+1}}}\right]\right\} \quad \text{for any }t>T_{\lambda,\varepsilon}.
		\end{equation*}
	\end{lemma}
 It is easy to check that initial data $u_0(x)\asymp e^{-\mu x^\beta}$ satisfy \eqref{genoo} in the regime $\beta<\frac{1}{\alpha+1}$. Thus, according to the above lemma, we obtain the following theorem.  
	
	\begin{theorem}\label{ei}
		Let $\alpha>0$ and $\beta>0$ be such that
		\begin{equation*}
			\beta<\frac{1}{\alpha+1}.
		\end{equation*}
		 Assume that the nonlinearity $f$ and the initial data $u_0$ satisfy Hypotheses \ref{fu} and \ref{ic}, respectively. Assume that there exists $x_0>0$ and $\mu>0$  such that
		\begin{equation*}
			u_0(x)\asymp e^{-\mu x^\beta} \quad \text{for any } x\ge x_0.
		\end{equation*}
		Then, for any $\lambda\in(0,1)$ and $\varepsilon>0$, there exists a time $T_{\lambda,\varepsilon}'$ such that \footnote{The notation $a\asymp_{\Lambda_1, \Lambda_2,...} b$ means that there exists a constant $C_{\Lambda_1, \Lambda_2,...}$, depending on some constants $\Lambda_1$, $\Lambda_2$,..., such that $C_{\Lambda_1, \Lambda_2,...} b\le a \le C^{-1}_{\Lambda_1, \Lambda_2,...} b$.}
		\begin{equation*}
		 x_\lambda(t) \asymp_{\lambda,\varepsilon,\mu} t^{\frac{1}{\beta(\alpha+1)}} \quad \text{for any } t>T_{\lambda,\varepsilon}'.
		\end{equation*}
	\end{theorem}
 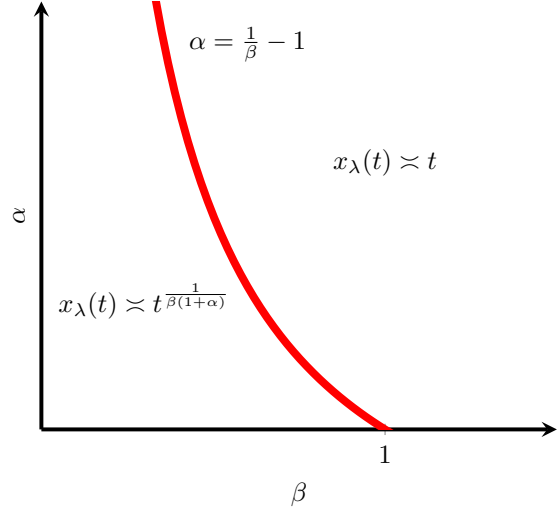
\begin{figure}
     \centering
\begin{tikzpicture}
    \begin{axis}[
    axis lines = left,
    xlabel = {$\beta$},
    ylabel = {$\alpha$},
    xmin = 0,
    xmax = 1.5,
    ymin = 0,
    ymax = 2,
    xtick = {1},
    xticklabels = {1},
    ytick = \empty,
    line width =1.5,
    ]
    \addplot[ 
    color = red,
    domain = 0.1:2,
    samples = 100,
    line width =3,
    ] {1/x-1}; 
\node[anchor= center] at (axis cs:0.3,0.6) {$x_\lambda(t)\asymp t^{\frac{1}{\beta(1+\alpha)}}$};
\node[anchor=center] at (axis cs:1,1.25) {$x_\lambda(t)\asymp t$};
\node[anchor=center] at (axis cs:0.6,1.8) {$\alpha=\frac{1}{\beta}-1$};
\end{axis}
\end{tikzpicture}
    \caption{The separation for sub-exponential decay initial data case.}
 \end{figure}
 
For initial data with algebraic tails, that is $u_0(x)\asymp x^{-\beta}$ for $\beta>0$, by Lemma \ref{gen}, we just obtain a rough estimate: 
\begin{equation*}
 C_1 e^{\frac{1}{\beta}((r-\varepsilon)(\alpha+1)t)^{\frac{1}{\alpha+1}}}   \le x_\lambda(t)\le C_2 e^{\frac{1}{\beta}((r+\varepsilon)(\alpha+1)t)^{\frac{1}{\alpha+1}}}
\end{equation*}
for some constants $C_1$ and $C_2$. Notice that the position of the level set depends strongly on the constant $\varepsilon$. Hence the estimate is not enough for such initial data.
\par

To get an exact estimate of the position of the level sets, we add a concavity assumption, that we believe not to be a huge restriction given that classical sub-exponentials are usually log-concave functions.
  \begin{lemma}\label{gen1}	
   Assume that the nonlinearity $f$ and the initial data satisfy Hypotheses \ref{fu} and \ref{ic}, respectively. Assume that $u_0$ is of class $C^2$ and  nonincreasing on $[\xi_0,+\infty)$ for some $\xi_0>0$, and
		\begin{equation*}
			\varphi_0'(x)=o(\varphi_0^{-\alpha}(x)) \text{  and  }\varphi_0''(x)=o(\varphi_0'(x))  \quad \text{as }x\to +\infty.
		\end{equation*}
   Assume that  
   \begin{equation}\label{au''}
				\varphi_0''(x)\le 0 \quad \text{for large }x.
			\end{equation}	
   Then, for any $\lambda\in(0,1)$, there are two constants $\underline C_\lambda>0$ and $\bar C_\lambda>0$ and a time $ T_{\lambda}$ such that
			\begin{equation}\label{gen1e}
				E_\lambda(t)\subset u^{-1}_0\left\{\left[\bar C_\lambda e^{-[r(\alpha+1)t]^{\frac{1}{\alpha+1}}},\underline C_\lambda e^{-[r(\alpha+1)t]^{\frac{1}{\alpha+1}}}\right]\right\} \quad \text{for any }t>T_{\lambda}.
			\end{equation}
		\end{lemma}	

 We point out that \eqref{au''} is used only in the proof of the lower bound. Our approach can be used to prove the exact result in the KPP situation of \cite{2009Fast}.
\par
  Equipped with the above lemma, we can get exact estimates for the level sets of the solution to equation \eqref{oeq1} with the algebraic decay initial data. We check the assumptions in Lemma \ref{gen1} and obtain the following theorem.
  \begin{theorem}\label{xbeta}			
Assume that the nonlinearity $f$ and the initial data satisfy Hypotheses \ref{fu} and \ref{ic}, respectively. Assume that there exist $x_0>1$ and $\beta>0$ such that
			\begin{equation*}
				u_0(x)\asymp \frac{1}{x^\beta} \quad \text{for any }x\ge x_0.
			\end{equation*}
   Then, for any $\lambda\in(0,1)$, there exists a time $ T_{\lambda}'$ such that
			\begin{equation*}
			x_\lambda(t) \asymp_\lambda e^{\frac{1}{\beta}[r(\alpha+1)t]^\frac{1}{\alpha+1}} \quad \text{for any }t> T_{\lambda}'.
			\end{equation*}
		\end{theorem}
Observe that when $\alpha=0$, one recovers the rate of the KPP situation \cite{2009Fast,henderson2016propagation}.
 For the degenerate monostable case, Alfaro \cite{alfaro2017slowing} shows that if $f(s)=s^{1+\alpha}(1-s)$ then  
 \begin{equation*}
     x_\lambda(t)\asymp t^{\frac{1}{\alpha\beta}}\quad \text{for } t \text{ large enough},
 \end{equation*}
 where $0<\alpha< \frac{1}{\beta}$.
 \par
 Thanks to Lemma \ref{gen1}, we can also get the following theorem for the initial data $u_0(x)\asymp (\ln x)^{-\beta}$.
 \begin{theorem}\label{lnxbeta}
			 Assume that the nonlinearity $f$ and the initial data $u_0$ satisfy Hypotheses \ref{fu} and \ref{ic}, respectively. Assume that there exists $x_0>e$ and $\beta>0$ such that
			\begin{equation}\label{u0lnb}
			u_0(x)\asymp (\ln x)^{-\beta}\quad \text{for any }x\ge x_0.
			\end{equation}
			Then,
			for any $x(t)\in E_\lambda(t)$, there exists a time $ T_{\lambda}''$, such that
			\begin{equation*}
				\ln x_\lambda(t) \asymp_\lambda e^{\frac{1}{\beta}[r(\alpha+1)t]^\frac{1}{\alpha+1}} \quad \text{for any }t> T_{\lambda}''.
			\end{equation*}
		\end{theorem}
  \begin{remark}
    One can obtain Lemma \ref{gen} and Theorem \ref{ei} under the weaker hypothesis
      \begin{equation*}
          f(s)\sim r \frac{s}{(1+|\ln s|)^\alpha} \quad \text{as } s\to 0^+,
      \end{equation*}
      for some $r>0$ and $\alpha>0$.
However, Lemma \ref{gen1} and Theorems \ref{xbeta} and \ref{lnxbeta} need crucially Hypothesis \ref{fu} with both precise bounds \eqref{fu1} and \eqref{fu2} for $f$. An insight can be easily seen in the proofs in Sections \ref{s3} and \ref{s4}.

  \par    
      Nevertheless, Lemmas \ref{gen} and \ref{gen1} and Theorems \ref{ei}, \ref{xbeta} and \ref{lnxbeta}, are true under Hypothesis \ref{fu} where \eqref{fu2} is replaced by
      \begin{equation*}
          f(s)\ge r \frac{s}{(1+|\ln s|)^\alpha}(1-Ks^\delta) \quad \text{for any } s\in(0,s_0],
      \end{equation*}
      for some $\delta>0$, $s_0\in(0,1)$ and $K\ge 0$. Their proofs are similar to the one we provide in Sections \ref{s3} and \ref{s4} but a bit messier so we have chosen to stick to $\delta = 1$ for the sake of readability.
  \end{remark}
  \par
 The rest of this paper is organized as follows. In Section \ref{s2}, we shall prove that the solution to equation \eqref{oeq1}, starting from an exponentially unbounded initial data, propagates at constant speed. In Section \ref{s3} and Section \ref{s4}, we provide the proof of the main results, Lemma \ref{gen} and Lemma \ref{gen1}, respectively.  In Section \ref{s5}, some numerical simulations shall be given to illustrate our main results.
 \par
This paper is the first part of our work on weakly monostable equations; a companion paper \cite{bouincovillezhang2023} with non-local dispersal follows. In this latter paper, we have proved the existence and nonexistence of traveling waves, and studied the effect of the tails of the dispersal kernel on the propagation rate. Exact rates of invasion have been provided for the sub-exponential and algebraic tails.
	
\section{Finite speed propagation: Proof of Theorem \ref{lin}}\label{s2}

In this section, we prove Theorem \ref{lin}: the level sets of the solution to \eqref{oeq1} moves at a constant speed. 
 \par
As in \cite[Theorem 2.3]{alfaro2017slowing}, we can also obtain that, for any $\lambda\in(0,1)$, there is a time $t_\lambda>0$ and $\Gamma>0$ such that 
\begin{equation}\label{enoe}
    \emptyset \neq E_\lambda(t)\subset(\Gamma t,+\infty) \quad \text{for any }t>t_\lambda.
\end{equation}  
Indeed, we consider the equation
\begin{equation}\label{lvequ}
\left\{
\begin{aligned}
&v_t-v_{xx}-r_1v^2(1-v)=0, &t>0,x\in\mathbb{R},\\
&v(0,x)=v_0\ge 0,
\end{aligned}\right.
\end{equation}
where the initial data $v_0(x)=\inf_{y\le 0}u_0(y)\mathds{1}_{(-\infty,0)}(x)$ and $r_1>0$ small enough so that $r_1s^2(1-s)\le f(s)$ for all $s\in(0,1)$.
According to \cite{zlatovs2005quenching}, the solution $v(t,x)$ to \eqref{lvequ} satisfies $\lim_{t\to \infty}\inf_{x\le \Gamma t}v(t,x)=1$ for some $\Gamma>0$. It follows from the comparison principle that propagation of $u(t,x)$ is at least linear, that is,
\begin{equation}\label{uto1}
\liminf_{t\to \infty\  x\le \Gamma t} u(t,x)=1.
\end{equation}
On the other hand, we can reproduce the proof of \cite[Theorem 1.1 part a]{2009Fast}, which does not require the KPP assumption, and get 
\begin{equation}\label{uto0}
    \lim_{x\to+\infty} u(t,x)=0 \quad \text{for any }t\ge 0.
\end{equation}
Thus, combining \eqref{uto1} and \eqref{uto0}, we can conclude \eqref{enoe}.

\par
  Inspired by \cite{alfaro2017slowing}, for the initial data with sub-exponential decay, we use a suitable shifted profile which construction now follows. Take $\alpha>0$ and $\beta>0$ such that
\begin{equation*}
	\beta\ge \frac{1}{\alpha+1}.
\end{equation*}
Let us define
\begin{equation}\label{lfw}
 w(z):=Me^{-\mu z^p}(\le 1)\quad \text{for } z\ge z_0:=\Big(\frac{\ln M}{\mu}\Big)^{\frac{1}{p}},
\end{equation}
where $p:= \frac{1}{\alpha+1}<1$ and $M>e$.
\begin{lemma}
Assume that $f$ satisfies Hypothesis \ref{fu}. Then, for any $M>e$, there is $c>0$ such that 
	\begin{equation*}
		w''(z)+cw'(z)+ f(w(z))\le 0, \quad \forall z\ge z_0.
	\end{equation*}
\end{lemma}
\begin{proof}
By definition of $w$, we have, for $z\ge z_0$,
\begin{equation*}
	    w'(z)=-\mu pz^{p-1}w(z)\quad\text{and}\quad w''=\left(\mu p(1-p)z^{p-2}+\mu^2 p^2 z^{2(p-1)}\right)w(z).
	\end{equation*}
Since $1-\ln M+\mu z^p\ge \frac{\mu}{2}z^p$ for any $z\ge (\tfrac{2(\ln M-1)}{\mu})^{\frac{1}{p}}$,  then we have, for all $z\ge z_1:=\max\{z_0, (\tfrac{2(\ln M-1)}{\mu})^{\frac{1}{p}}\}$,
	\begin{equation*}
		\begin{aligned}
			w''(z)+cw'(z)+f(w(z))&\le \mu w(z)\left(\frac{\mu p^2}{z^{2(1-p)}}+\frac{p(1-p)}{z^{2-p}}-\frac{cp}{z^{1-p}}+\frac{r}{\mu(1-\ln M+ \mu z^p)^{\alpha}}\right)\\
			&\le \mu w(z)\left(\frac{\mu p^2}{z^{2(1-p)}}+\frac{p(1-p)}{z^{2-p}}-\frac{cp-\frac{ 2^\alpha r}{\mu ^{\alpha+1}}}{z^{1-p}}\right).
		\end{aligned}
	\end{equation*}
 Choosing $c>\tfrac{2^\alpha r}{\mu^{\alpha+1} p}$, the above is nonpositive for $z$ large enough, say $z\ge z_2$.  
	On the other hand, for the remaining region $z_0\le z\le z_2$, we have
	\begin{equation*}
		\begin{aligned}
			w''(z)+cw'(z)+f(w(z))&\le \mu w(z)\left(\frac{\mu p^2}{z^{2(1-p)}}+\frac{p(1-p)}{z^{2-p}}-\frac{cp}{z^{1-p}}+\frac{r}{\mu(1-\ln M+ \mu z^p)^{\alpha}}\right)\\
			&\le \mu w(z)\left(\frac{\mu p^2}{z_0^{2(1-p)}}+\frac{p(1-p)}{z_0^{2-p}}-\frac{cp}{z_2^{1-p}}+\frac{r}{\mu(1-\ln M+ \mu z_0^p)^{\alpha}}\right),
		\end{aligned}
	\end{equation*}
by taking $c$ large enough so that the above is nonpositive. 
\end{proof}
Equipped with the above lemma, we can construct a supersolution to \eqref{oeq1}. Let $M=\max\{e,e^{x_0^\beta}||u_0||_\infty\}$ and define
\begin{equation*}
v(t,x):=\left\{ 
\begin{aligned}
&w(x-x_0-ct), &x>ct+ x_0+z_0,\\
&1,  &x\le ct+ x_0+z_0,
 \end{aligned}\right.
\end{equation*}
where $c$ and $x_0$ is from the above lemma and \eqref{x0} respectively. 
We claim that $v(t,x)$ is a supersolution for \eqref{oeq1} for any $x\in\mathbb{R}$ and $t>0$.
Indeed, it is enough to check it when $v(t,x)<1$, that is, $x>ct+x_0+z_0$. It follows from the above lemma that 
\begin{equation*}
	v_t-v_{xx}-f(v)\ge-(cw'-w''+r\frac{w}{(1-\ln w)^\alpha})\ge 0.
\end{equation*}
 For $x>x_0+z_0$, since $p=\frac{1}{\alpha+1}\le \beta$ and \eqref{x0},  then we have
\begin{equation*}
	v(0,x)=Me^{-\mu (x-x_0)^p}\ge e^{x_0^\beta} ||u_0||_\infty e^{-\mu x^\beta} \ge u_0(x).
\end{equation*}
On the other hand, since $u_0\le 1$, for $x\le x_0+z_0$, we have $v(0,x)=1\ge u_0(x)$. 
\par
In the regime $\beta\ge \frac{1}{\alpha+1}$, the comparison principle then implies that for all $t>0$ and  $x\in\mathbb{R}$, we have
\begin{equation*}
	u(t,x)\le v(t,x)\le w(x-x_0-ct). 
\end{equation*}
Therefore, for any $\lambda\in(0,1)$, there is $T_\lambda$ large enough such that for all $t>T_\lambda$, we have
\begin{equation*}
	x_\lambda(t)\le x_0+\Big(\frac{1}{\mu}\ln \frac{M}{\lambda}\Big)^{\frac{1}{p}}+ct\le (c+1)t,
\end{equation*} 
which gives the upper bound of \eqref{linineq}. Together with \eqref{enoe},  the proof of Theorem \ref{lin} is complete.
\section{The acceleration regime: Proof of Lemma \ref{gen}}\label{s3}
In this section, we prove Lemma \ref{gen}: the level sets of solution to the equation  \eqref{oeq1} with  front-like initial data that is sub-exponentially unbounded move by accelerating, and the locations of the level sets are expressed in terms of the decay of the initial data. 
\par 
The long-time behaviour of the solution to the Cauchy problem \eqref{oeq1} is captured approximately by the ODE
\begin{equation}\label{ode}
	\left\{\begin{aligned}
		& w_t=\rho\frac{w}{(1-\ln w)^\alpha},&t>0, x\in\mathbb{R},\\
		& w(0,x)=u_0(x)\ge 0,&x\in \mathbb{R},
	\end{aligned}\right.
\end{equation}
where $\rho>0$ is to be determined. We solve the above ODE and obtain
\begin{equation}\label{wxt}
	w(t,x)=\exp{\left\{1-\left[(1+\varphi_0(x) )^{\alpha+1}-\rho(\alpha+1)t\right]^{\frac{1}{\alpha+1}}\right\}},
\end{equation}
where $\varphi_0$ is defined by \eqref{lvar}.
Notice that $w(t,x)\ge w(0,x)=u_0(x)$ since $w(x,\cdot)$ is increasing for each $x\in\mathbb{R}$.
Let us define 
\begin{equation}\label{lu0}
 x_0(t):=\sup\Big\{x\in\mathbb{R}:u_0(x)=\exp \Big(1-\big(\rho(\alpha+1)t+1\big)^{\frac{1}{\alpha+1}}\Big)\Big\}.
\end{equation}
Observe that $w(t,x_0(t))=1$ and $0<w(t,x)\le 1$ for $x\ge x_0(t)$. For any $x\ge x_0(t)$ and $t>0$, we have
\begin{equation}\label{wx}
    w_x=-\frac{w}{(1-\ln w)^\alpha}\varphi_0'(1+\varphi_0)^\alpha,
\end{equation}
and 
\begin{equation}\label{wxx}
		w_{xx}=\frac{w}{(1-\ln w)^{\alpha}}\Big\{(\varphi_0')^2 (1+\varphi_0)^{2\alpha}\Big((1-\ln w)^{-\alpha}
		+\alpha (1-\ln w)^{-(\alpha+1)}-\alpha(1+\varphi_0)^{-(\alpha+1)} \Big)-\varphi_0''(1+\varphi_0)^\alpha \Big\}.
\end{equation}
For $w_{xx}$, we have the following estimate.
\begin{lemma}\label{de} Let $u_0$ such that $\varphi_0=-\ln u_0$ satisfies $\varphi_0'=o(\varphi_0^{-\alpha}) \text{ and }\varphi_0''=o(\varphi_0') \text{ as }x\to +\infty$.
 Then,  for any small $\varepsilon>0$, there exists $t^\#>0$, depending on $\varepsilon$, such that
	\begin{equation}\label{de1}
		|w_{xx}|<\varepsilon\frac{w}{(1-\ln w)^{\alpha}}\quad \text{for any }x\ge x_0(t) \text{ and } t\ge t^\#.
	\end{equation}	
\end{lemma}
\begin{proof}
	
	Since $0<u_0(x)\le w(t,x)\le 1$, we have 
 \begin{equation*}
     0<(1+\varphi_0)^{-(\alpha+1)}\le (1-\ln w)^{-(\alpha+1)}\le 1.
 \end{equation*}
 It follows from $0< w\le 1$ for all $x\ge x_0(t)$ and $t>0$  that, for any $x\ge x_0(t)$ and $t> 0$, we have
	\begin{equation}\label{Cw}
		0<(1-\ln w)^{-\alpha}+\alpha(1-\ln w)^{-(\alpha+1)}-\alpha(1+\varphi_0)^{-(\alpha+1)}<2.
	\end{equation}
In view of the definition \eqref{lu0} of $x_0(t)$, since $u_0$ is nonincreasing and $\lim_{x\to +\infty} u_0=0$, we have $x_0(t)\to +\infty$ as $t\to \infty$. For any small $\varepsilon>0$, it follows from the assumption on $\varphi_0'(x)$ that there exists $t^\#>0$ such that for $x>x_0(t)$ and $t\ge t^\#$,  we have
	\begin{equation*}
		|\varphi_0'(x)(1+\varphi_0(x))^\alpha|<\sqrt{\frac{\varepsilon}{4}}.
	\end{equation*}
	On the other hand, it follows from $\varphi_0''(x)=o(\varphi_0'(x))$ that there exists $X'$ such that  for $x>x_0(t)\ge X'$ and $t\ge t^\#$, up to enlarge $t^\#$ if necessary, we have
	\begin{equation*}
		|\varphi_0''(x)(1+\varphi_0(x))^{\alpha}|<\varepsilon/2.
	\end{equation*}
	Therefore, by collecting the above estimates, we have, for any $x\ge x_0(t)$ and $t\ge t^\#$,
	\begin{equation*}
		\begin{aligned}
			|w_{xx}(t,x)|&\le \frac{w(t,x)}{(1-\ln w(t,x))^{\alpha}} \Big\{(\varphi_0'(x)(1+\varphi_0(x))^\alpha)^2\Big((1-\ln w(t,x))^{-\alpha}\\
   &+\alpha \big(1-\ln w(t,x)\big)^{-(\alpha+1)}-\alpha(1+\varphi_0(x))^{-(\alpha+1)}\Big)+|\varphi_0''(x)(1+\varphi_0(x))^{\alpha}|\Big\}\\
			&\le \frac{w(t,x)}{(1-\ln w(t,x))^{\alpha}}(\frac{\varepsilon}{4}\times 2+\frac{\varepsilon}{2})=\varepsilon\frac{w(t,x)}{(1-\ln w(t,x))^{\alpha}},
		\end{aligned}
	\end{equation*}
	which gives the estimate \eqref{de1}. This completes the proof.
 
\end{proof}
Here we present a lemma, which will play a key role in the proof of Lemma \ref{gen1}.
\begin{lemma}\label{lem22}
Let $u_0$ such that $\varphi_0=-\ln u_0$ satisfies $\varphi_0'=o(\varphi_0^{-\alpha}) \text{ and }\varphi_0''=o(\varphi_0') \text{ as }x\to +\infty$.
   Then there is $ t^1>0$ such that, for $x\ge x_0(t)$ and $t\ge t^1$, we have
\begin{equation}\label{wxwxx}
    w_x+w_{xx}\le 0.
\end{equation}  
\end{lemma}

\begin{proof}
By the assumptions $\varphi_0'(x)=o(\varphi_0^{-\alpha}(x))$ and $\varphi_0''(x)=o(\varphi_0'(x))$, there exists $X_0$ such that for all $x\ge X_0$, we have
\begin{equation*}
    \varphi_0'(1+\varphi_0)^\alpha\le \frac{1}{4} \text{ and }|\varphi_0''|\le \frac{1}{2}\varphi_0'.
\end{equation*}
Since $x_0(t)\to +\infty$ as $t\to \infty$, there is $t^1>0$ such that $x_0(t)>X_0$ for all $t\ge t^1$. In view of the definition \eqref{lvar} of $\varphi$, since $u_0$ is a nonincreasing function, then $\varphi'\ge 0$.
It then follows from  \eqref{wx}, \eqref{wxx} and \eqref{Cw}  that we have, for all $x\ge x_0(t)$ and $t\ge t^1$,

\begin{align*}
w_x+w_{xx} &\le\frac{w}{(1-\ln w)^\alpha}  (1+\varphi_0)^\alpha\Big(\varphi_0'(-1+2\varphi_0' (1+\varphi_0)^{\alpha}) +|\varphi_0''|\Big)\\
  &\le \frac{w}{(1-\ln w)^\alpha}  (1+\varphi_0)^\alpha\Big(-\frac{1}{2}\varphi_0'+\frac{1}{2}\varphi_0'\Big)=0.
  \end{align*}
This completes the proof.
\end{proof}

\subsection{The upper bound}
In this subsection, we prove the upper bound of the level sets in Lemma \ref{gen} by constructing an accurate supersolution.
\par 
We define
\begin{equation*}
	m(t,x)=
 \left\{
 \begin{aligned}
     &w(t+t^\#,x),  &x\ge x_0(t+t^\#),\\
     &1,& x<x_0(t+t^\#),
 \end{aligned}\right.
\end{equation*}
 where $t^\#$ is defined in Lemma \ref{de}. Observe that $m(t,x)$ is well defined for all $t\ge 0$ and all $x\in \mathbb{R}$, and $0<m(t,x)\le 1$.
\par
Let $\varepsilon>0$ be given and define 
\begin{equation}\label{rhou}
	 \rho=r+\frac{\varepsilon}{2}.
\end{equation}
Now, we prove that $m$ is a supersolution of equation \eqref{oeq1}.
\begin{lemma}\label{sus}
Let $u_0$ such that $\varphi_0=-\ln u_0$ satisfies $\varphi_0'=o(\varphi_0^{-\alpha}) \text{ and }\varphi_0''=o(\varphi_0') \text{ as }x\to +\infty$.
  Then $m(t,x)$ is a supersolution to equation \eqref{oeq1}  for all $t>0$ and $x\in\mathbb{R}$.
\end{lemma}
\begin{proof}
To prove $m$ is a supersolution, we need to check that $m_t-m_{xx}-f(m)\ge 0$ for all $t>0$ and $x\in\mathbb{R}$.
For $x<x_0(t+t^\#)$ and $t>0$, since $m_t=m_{xx}=f(m)=0$,  we have 
\begin{equation*}
    m_t(t,x)-m_{xx}(t,x)-f(m(t,x))=0.
\end{equation*}
On the other hand, for all $x\ge x_0(t+t^\#)$ and $t>0$, by the definitions of $m$ and $w$, we have
\begin{equation*}
    m_t(t,x)= w_t(t+t^\#,x)= \rho \frac{w(t+t^\#,x)}{(1-\ln w(t+t^\#,x))^{\alpha}}=\Big(r+\frac{\varepsilon}{2}\Big) \frac{w(t+t^\#,x)}{(1-\ln w(t+t^\#,x))^{\alpha}}.
\end{equation*}
Thus, by Lemma \ref{de} and Hypothesis \ref{fu}, for all $x\ge x_0(t+t^\#)> x_0(t)$ and $t>0$, we obtain  
\begin{align*}
    m_t(t,x)-m_{xx}(t,x)&-f(m(t,x))= w_t(t+t^\#,x)-w_{xx}(t+t^\#,x)-f(w(t+t^\#,x))\\
   & \ge\Big(r+\frac{\varepsilon}{2}\Big) \frac{w(t+t^\#,x)}{(1-\ln w(t+t^\#,x))^{\alpha}}- \frac{\epsilon}{2}\frac{w(t+t^\#,x)}{(1-\ln w(t+t^\#,x))^{\alpha}}-r\frac{w(t+t^\#,x)}{(1-\ln w(t+t^\#,x))^\alpha}= 0.
\end{align*}
This completes the proof.
\end{proof}
In view of the definition of $m$, for $x<x_0(t+t^\#)$, since $u_0\le 1$,  we have $m(0,x)=1\ge u_0(x)$.
For $x\ge x_0(t+t^\#)$,  since $w(\cdot,x)$ is nondecreasing for each $x\in\mathbb{R}$, we have $m(0,x)=w(t^\#,x)\ge w(0,x)= u_0(x)$. Thus, $m(0,x)\ge u_0(x)=u(0,x)$ for all $x\in\mathbb{R}$.
Equipped with Lemma \ref{sus}, it then follows from the comparison principle that 
\begin{equation}\label{m}
	m(t,x)\ge u(t,x) \quad \text{for all }t>0 \text{ and }x\in\mathbb{R}.
\end{equation} 
Now, we prove the upper bound in Lemma \ref{gen}. 
\begin{proof}[The proof of the upper bound]
	We need to prove that, for any $\lambda\in (0,1)$ and large time $t$, we have 
	\begin{equation}\label{upperb}
		E_\lambda(t)\subset u^{-1}_0\left\{\left[e^{-[(r+\varepsilon)(\alpha+1)t]^{\frac{1}{\alpha+1}}},1\right]\right\}.
	\end{equation}
	By \eqref{m} and the definition of  $m$, we have 
	\begin{equation*}
		u(t,x)\le m(t,x)\le w(t+t^\#,x) \quad \text{for all }t>0 \text{ and }x\in\mathbb{R}.
	\end{equation*}
	Let us pick a $y\in E_\lambda(t)$, then $w(t+t^\#,y)\ge\lambda$.
	It follows that,
 by the definitions of $w$ and $\varphi_0$, 
    we have 
    \begin{equation*}
        w(t+t^\#,y)= \exp \Big\{1-[(1-\ln u_0(y))^{\alpha+1}-\rho(\alpha+1)(t+t^\#)]^{\frac{1}{\alpha+1}} \Big\}\ge\lambda,
    \end{equation*}
    whence 
    \begin{equation*}
        u_0(y)\ge \exp\Big\{1-\Big[\rho(\alpha+1)(t+t^\#)+(1-\ln \lambda)^{\alpha+1}\Big]^{\frac{1}{\alpha+1}}\Big\}.
    \end{equation*}

 Since $\rho= r+\frac{\varepsilon}{2}$, there is a time $\bar t_{\lambda,\varepsilon}> 0$ such that
	\begin{equation}\label{u1}
		u_0(y)\ge e^{-[(r+\varepsilon)(\alpha+1)t]^{\frac{1}{\alpha+1}}}\quad \text{for any }t\ge \bar t_{\lambda,\varepsilon},
	\end{equation}
	which gives \eqref{upperb}. This completes the proof.
\end{proof}

\subsection{The lower bound}\label{lalb}
In this subsection, we explore the lower bound of the level sets of a solution to \eqref{oeq1} by constructing an adequate subsolution.\par
Let $\varepsilon>0$ be given. We take
\begin{equation}\label{rho}
	\max\Big\{r-\frac{\varepsilon}{2},\frac{3}{4}r\Big\}<\rho<r. 
\end{equation}
Let us define the function $g(y):=y(1-My)$ with $M>0$. Notice that
\begin{equation*}
    0\le g(y)\le g\Big(\frac{1}{2M}\Big)=\frac{1}{4M}, \quad \forall y\in \left[0, \frac{1}{2M}\right].
\end{equation*}
We define 
\begin{equation*}
    x_M(t):=\sup \Big\{x\in\mathbb{R}:u_0(x)=\exp\big\{1-\big((1+\ln\left( 2M\right))^{\alpha+1}+\rho(\alpha+1)t\big)^{\frac{1}{\alpha+1}} \big\}\Big\}\ge x_0(t),
\end{equation*}
where $x_0(t)$ is defined by \eqref{lu0}.
Observe that $w(t,x_M(t))=\frac{1}{2M}$ and $w(t,x)<\frac{1}{2M}$ for all $x>x_M(t)$.
Let us define
\begin{equation*}
	\zeta:=\inf_{x\in(-\infty,\xi_1)}u_0(x),
\end{equation*}
where $\xi_1:=\max\{\xi_0, x_0(0)\}$.
  Notice that $\zeta\in(0,1]$ according to Hypothesis \ref{ic}, and that $u_0$ is non-increasing on $[\xi_1,+\infty)$. 
We select large enough $M>0$ so that
\begin{equation*}
      M\ge M_0:=\max\left\{\frac{1}{2\zeta},\frac{1}{4s_0}\right\}.
  \end{equation*}
Then, by $u_0(x_M(0))=\frac{1}{2M}<\zeta\le u_0(\xi_1)$, we have $x_M(0)> \xi_1$.

Let us define
\begin{equation}\label{lvxt}
v(t,x):=
\left\{
\begin{aligned}
    &\frac{1}{4M}, &x\le x_M(t), \\
    &g(w(t,x)), &x>x_M(t).
\end{aligned}\right.
\end{equation}
Since $M\ge \frac{1}{4s_0}$, then we have $0<v(t,x)\le s_0$ for all $t\ge 0$ and $x\in\mathbb{R}$.
\begin{lemma}
Let $u_0$ such that $\varphi_0=-\ln u_0$ satisfies $\varphi_0'=o(\varphi_0^{-\alpha}) \text{ and }\varphi_0''=o(\varphi_0') \text{ as }x\to +\infty$. Then there exists large enough $M>0$ such that $v(t,x)$ is a subsolution to equation \eqref{oeq1} for all $t>0$ and $x\in\mathbb{R}$.
\end{lemma}
\begin{proof}

\par In view of the definition \eqref{lvxt} of $v$, we obtain
\begin{equation*}
    v_t(t,x)=\rho \frac{w(t,x)}{(1-\ln w(t,x))^\alpha}\Big(1-2Mw(t,x)\Big)^+.
\end{equation*}
It then follows that
\begin{equation}\label{llmt}
v_t(t,x)\le \left\{
\begin{aligned}
    &0,& x\le x_M(t),\\
    &\rho \frac{w(t,x)}{(1-\ln w(t,x))^\alpha}\Big(1-2Mw(t,x)\Big), &x>x_M(t).
\end{aligned}\right.
\end{equation}
\par Since $\frac{1}{(1+y)^\alpha}\ge 1-\alpha y$ for any $y\ge0$, then, for any $x>x_M(t)$, we have
\begin{equation*}
    1\ge \Big(\frac{1-\ln w}{1-\ln v}\Big)^\alpha=\Big(\frac{1}{1-\frac{\ln(1-Mw)}{1-\ln w}}\Big)^\alpha\ge 1+\alpha \frac{\ln(1-Mw)}{1-\ln w}.
\end{equation*}
It follows from the inequality $\ln(1-y)\ge - \tilde c y$  for any $y\in (0,\frac{1}{2})$, where $\tilde c= 2\ln 2$, that, for any $x>x_M(t)$,  we have
\begin{equation*}
    1\ge \Big(\frac{1-\ln w}{1-\ln v}\Big)^\alpha\ge 1-\alpha \tilde c  M\frac{w}{1-\ln w},
\end{equation*}
from $0<Mw<\frac{1}{2}$.
Thus, since $f(s)\ge r \frac{s}{(1-\ln s)^\alpha}(1-Ks)$ for $s\in(0,s_0]$ and $0<v(t,x)\le s_0$ for all $t\ge 0$ and $x\in\mathbb{R}$, when $x> x_M(t)$, we have
\begin{equation}\label{llfu1}
\begin{aligned}
    f(v(t,x))&\ge r \frac{v(t,x)}{(1- \ln v(t,x))^\alpha}(1-Kv(t,x))\\
    &\ge r\frac{w(t,x)(1-M w(t,x))}{(1-\ln w(t,x))^\alpha}\Big(\frac{1-\ln w(t,x)}{1-\ln v(t,x)}\Big)^\alpha\big(1- Kw(t,x)  \big)\\
    &\ge r \frac{w(t,x)}{(1-\ln w(t,x))^\alpha}\Big(1-(M+K)w(t,x)-\alpha \tilde c M\frac{w(t,x)}{1-\ln w(t,x)}-\alpha \tilde c KM^2 \frac{w^{3}(t,x)}{1-\ln w(t,x)}\Big).
    \end{aligned}
\end{equation}
On the other hand, when $x\le x_M(t)$, we have
\begin{equation}\label{llfu2}
	f(v(t,x))\ge 0,
\end{equation}
thanks to $v\in(0,1)$.

  \par Now, let us estimate the value $v_{xx}$.  In view of the definition \eqref{lvxt} of $v$, we have $v_{xx}(t,x)=0$ for $x\le x_M(t)$ and 
  \begin{equation*}
      v_{xx}(t,x)=\left(1-2M w(t,x)\right)w_{xx}(t,x)-2 M  w_x^2(t,x) \quad \text{for }x >x_M(t).
  \end{equation*}
In view of \eqref{wx} and \eqref{wxx}, we get 
\begin{equation*}
    v_{xx}(t,x)\ge -\varphi_0''(x)(1+\varphi_0(x))^\alpha (1-2M w(t,x))
\frac{w(t,x)}{(1-\ln w(t,x))^\alpha}-2M \left(\varphi'_0(x)(1+\varphi_0(x))^\alpha\right)^2\frac{w^2(t,x)}{(1-\ln w(t,x))^{2\alpha}}.
\end{equation*}
Since $\varphi''_0(x)=o(\varphi'(x))$ and  $\varphi_0'(x)=o(\varphi_0^{-\alpha}(x))$ as $x\to +\infty$,  there exists $X_1>\xi_0$ such that 
\begin{equation}
    \varphi_0''(x)(1+\varphi_0(x))^\alpha\le \frac{r-\rho}{4}\quad \text{and}\quad\varphi'_0(x)(1+\varphi_0(x))^\alpha\le \frac{\sqrt{r-\rho}}{2},
\end{equation}
 as $x\to +\infty$. In view of the definition of $x_M$, we take $M$ large enough, say $M>M_1\ge M_0$, such that $x_M(t)>X_1$ for all $t>0$. Thus, for $x>x_M(t)$, we have
\begin{equation}\label{llmxx}
		\begin{aligned}
			v_{xx}&\ge-\frac{r-\rho}{4} \frac{w}{(1-\ln w)^{\alpha}}-\frac{r-\rho}{2} M\frac{w^{2}}{(1-\ln w)^{2\alpha}}\\
   &\ge -\frac{r-\rho}{2} \frac{w}{(1-\ln w)^{\alpha}},
		\end{aligned}
	\end{equation}
 thanks to $0<w\le\frac{1}{2M}$.
 \par
Collecting \eqref{llmt}, \eqref{llfu1}, \eqref{llfu2} and \eqref{llmxx}, for $x\le x_M(t)$, we obtain
\begin{equation}\label{laamtmxxfm1}
\begin{aligned}
    (v_t-v_{xx}-f(v))(t,x)\le 0,
    \end{aligned}
\end{equation}
whereas, for $x>x_M(t)$,
\begin{equation*}
\begin{aligned}
    (v_t-v_{xx}-f(v))(t,x)\le&\frac{w(t,x)}{(1-\ln w(t,x))^\alpha}\Big(-\frac{1}{2}(r-\rho)
    +\Big((-2\rho+r)M+rK\Big)w(t,x)\\
    &+r\alpha \tilde c  M\frac{w(t,x)}{1-\ln w(t,x)}+\alpha \tilde c K M^2 \frac{w^{3}(t,x)}{1-\ln w(t,x)}\Big).
    \end{aligned}
\end{equation*}
For $0<w\le \frac{1}{2M}$, when $M\ge M_2:=\max\left\{ \frac{1}{2}\exp\left(2\alpha \tilde c(1+\frac{1}{r})-1\right),\frac{K}{4}\right\}$ , we have
\begin{equation*}
    \frac{r+KMw^2}{1-\ln w}\le \frac{r+1}{1+\ln(2M)}\le \frac{r}{2\alpha \tilde c}.
\end{equation*}
Thus, for any $x>x_M(t)$, we take $M$ large enough, say
\begin{equation*}
    M\ge \tilde M:=\max\Big\{\frac{rK}{2\rho-\frac{3}{2}r},M_1,M_2\Big\},
\end{equation*}
so that 
\begin{equation*}
    r\alpha \tilde c  M\frac{w(t,x)}{1-\ln w(t,x)}+\alpha \tilde c K M^2 \frac{w^{3}(t,x)}{1-\ln w(t,x)}=\alpha \tilde c  M\frac{r+KMw^2(t,x)}{1-\ln w(t,x)} w(t,x)\le \frac{1}{2}rMw (t,x),
\end{equation*}
whence
\begin{equation}\label{laamtmxxfm2}
    (v_t-v_{xx}-f(v))(t,x)
    \le\frac{w^{2}(t,x)}{(1-\ln w(t,x))^\alpha}\Big((-2\rho+\frac{3}{2}r)M+rK\Big)\le 0,
\end{equation}
thanks to $\rho>\frac{3}{4}r$.
This completes the construction of the subsolution $ v(t,x)$.
\end{proof}
 In view of \eqref{lvxt}, we notice that:
\begin{itemize}
    \item when $x>x_M(0)$, we have $v(0,x)=g(w(0,x))\le w(0,x)=u_0(x)$;
    \item when $\xi_1\le x\le x_M(0)$, since $u_0$ is nonincreasing on $[\xi_1,+\infty)$, we have $v(0,x)=\frac{1}{4M}<\frac{1}{2M}\le u_0(x)$;
    \item when $x< \xi_1$, we have $v(0,x)=\frac{1}{4M}<\frac{1}{2M}<\zeta\le u_0(x)$.
\end{itemize}
 Thus, we obtain $v(0,x)\le u_0(x)=u(0,x)$ for all $x\in\mathbb{R}$.
As a consequence, the maximum principle yields
\begin{equation}\label{v}
	v(t,x)\le u(t,x)  \quad \text{for all }t\ge 0 \text{ and }x\in\mathbb{R}. 
\end{equation}
Now, we prove the lower bound in Lemma \ref{ei}. 
\begin{proof}[The proof of the lower bound]
	\par
	Firstly, we prove it for small $\theta$. Let us fix 
	\begin{equation*}
		0<\theta <\frac{1}{4M}.
	\end{equation*}
	We define the level set of $w(t,x)$ as 
	\begin{equation*}
		F_\theta(t):=\{x\in\mathbb{R},\ w(t,x)=\theta\}.
	\end{equation*}
	 Recall that $E_\theta(t)$ is not empty for $t>t_\theta$ 
  . It follows from Hypothesis \ref{ic} that there exists a time $t'_\theta>t_\theta$ 
  such that, for any $t>t'_\theta$, the closed set $F_\theta(t)$  is nonempty. For any $t\ge t'_\theta$, denote
	\begin{equation*}
		y_\theta(t):=\min F_\theta(t).
	\end{equation*}
	Then the function $y_\theta: [t'_\theta,+\infty)\to \mathbb{R}$ is nondecreasing and left-continuous. In addition, since $u_0$ is nonincreasing, for all points $t\ge t'_\theta$ where the function $y_\theta$ is discontinuous, there exist $a<b$ such that 
 \begin{equation*}
     u_0=\exp\Big\{1-\left[\rho(\alpha+1)t+(1-\ln \theta)^{\alpha+1}\right]^{\frac{1}{\alpha+1}}\Big\} \quad\text{on }[a,b];
 \end{equation*}
if $[a,b]$ denotes the largest such interval, then $a=y_\theta(t)$ and $b=y_\theta(t^+)=\lim_{s\to t, s>t}y_\theta(s)$.
\par
	We claim that 
	\begin{equation*}
		\inf_{\bar\Omega}u>0,
	\end{equation*}
	where $\Omega$ is an open set defined by
	\begin{equation*}
		\Omega:=\left\{(t,x), t>0,x<y_\theta(t)\right\}.
	\end{equation*}
	Let us evaluate $u(t,x)$ on the boundary $\partial \Omega$. 
	$\partial \Omega$ consists in two parts: 
	\begin{itemize}
		\item[(1)]  $\{ t'_\theta\}\times (-\infty, y_\theta( t_\theta'^+)]$;
		\item[(2)] $\big\{(t,x)|t>t'_\theta\quad \text{and} \quad x\in[y_\theta(t),y_\theta(t^+)]\big\}$.
	\end{itemize}
	For the first part, $u(t'_\theta,\cdot)$ is continuous, positive, and $\liminf_{x\to -\infty}u(t'_\theta, x)>0$. 
 Thus, we have
	\begin{equation*}
		\inf_{x\in(-\infty, y_\theta(t_\theta'^+)]}u(t'_\theta,x)>0.
	\end{equation*}
	For the second part, if $t>t'_\theta$ and $x\in[y_\theta(t),y_\theta(t^+)]$, then $w(t,x)=\theta$, whence 
	\begin{equation*}
		u(t,x)\ge \theta-M\theta^{2}>0.
	\end{equation*}
 
	As a consequence, $\Theta:=\inf_{\partial\Omega} u>0$. Since $\Theta>0$ is a subsolution of equation \eqref{oeq1}, the comparison principle yields
	\begin{equation}\label{Theta}
		u(t,x)\ge\Theta\quad \text {for all }  x\in \bar\Omega.
	\end{equation}
	Let us pick any $x\in E_\lambda(t)$ for any $\lambda\in(0,\Theta)$. Then 
	\begin{equation*}
		x>y_\theta(t^+)\ge y_\theta(t) \quad\text{for any } t\ge t'_\theta.
	\end{equation*}  
	Since $\rho>r-\frac{\varepsilon}{2}$, then there exists a time $\underline t'_{\lambda,\varepsilon}>t'_\theta$ such that
	\begin{equation}\label{l1}
		u_0(x)\le u_0(y_\theta(t))=\exp\left\{1-\left[\rho(\alpha+1)t+(1-\ln \theta)^{\alpha+1}\right]^{\frac{1}{\alpha+1}}\right\}\le e^{-[(r-\varepsilon)(\alpha+1)t]^{\frac{1}{\alpha+1}}}
	\end{equation}
	for any $t>\underline t'_{\lambda,\varepsilon}$, which gives the lower bound for small $\lambda$.
	\par  Let us prove the lower bound for any $\lambda\in(0,1)$. Let $\lambda\in [\Theta,1)$ be given. Denote by $u_\theta$  the solution to \ref{oeq1} with initial data
	\begin{equation}\label{utheta}
		u_{\theta,0}:=\left\{
		\begin{aligned}
			&\Theta, & x\ge -1,\\
			&-\Theta x, & -1<x<0,\\
			&0,&x\ge 0.
		\end{aligned}
		\right.
	\end{equation}
In view of \eqref{uto1}, we can also obtain that, for some $\gamma_1>0$, we have $\lim_{t\to \infty}\inf_{x\le \gamma_1 t}u_\theta(t,x)=1$.
	\par 
	There exists a time $\underline t''_\lambda>0$ such that 
	\begin{equation}\label{t''}
		u_\theta( \underline t''_\lambda,x)>\lambda \quad \text{for all }x\le0.
	\end{equation}
	Furthermore, by \eqref{Theta} and \eqref{utheta}, we have
	\begin{equation*}
		u(t,x)\ge u_{\theta,0}(x-y_\theta(T)) \quad \text{for any }x\in\mathbb{R} \text{ and } T\ge0.
	\end{equation*} 
	It follows from the comparison principle that 
	\begin{equation*}
		u(T+t,x)\ge u_\theta(t,x-y_\theta(T)).
	\end{equation*}
	By \eqref{t''}, we obtain
	\begin{equation*}
		u(T+\underline t''_\lambda,x)>\lambda, \quad \text{for all } x\le y_\theta(T)\text{ and } T\ge0.
	\end{equation*}
	Therefore there exists a time $ \underline t_{\lambda,\varepsilon}>\max(\underline t'_{\lambda,\varepsilon},\underline t''_\lambda)$ such that
	\begin{equation}\label{l2}
        \begin{aligned}
		u_0(x_\lambda(t))&<u_0(y_\theta(t-\underline t''_\lambda))\\
  &=\exp\left\{1-\left[\rho(\alpha+1)(t-\underline t''_\lambda)+(1-\ln \lambda)^{\alpha+1}\right]^{\frac{1}{\alpha+1}}\right\}\\
  &< e^{-[(r-\varepsilon)(\alpha+1)t]^{\frac{1}{\alpha+1}}}
        \end{aligned}
	\end{equation}
	for any $t>\underline t_{\lambda,\varepsilon}$, which gives the lower bound. This completes the proof.
\end{proof}
Let $T_{\lambda,\varepsilon}=\max\{\bar t_{\lambda,\varepsilon}, \underline t_{\lambda,\varepsilon}\}$. Thus, combining \eqref{u1}, \eqref{l1} and \eqref{l2}, the proof of Lemma \ref{gen} is complete.

\section{A more precise bound: Proof of Lemma \ref{gen1}} \label{s4}
In this section, we give a more precise bound for the level sets of the solution to the equation \eqref{oeq1}.  
\subsection{The upper bound} 
We derive a more precise upper bound by translating the spatial variables in $w$, so that the supersolution can approximate the solution to \eqref{oeq1} more accurately.
\par
Let us define
\begin{equation*}
    m(t,x)= \left\{\begin{aligned}
    &w(t+t^1,x-t),  &x\ge x_0(t+t^1)+t,\\
    &1, &x<x_0(t+t^1)+t,
    \end{aligned}\right.
\end{equation*}
where $w(t,x)$ is defined by \eqref{wxt} with $\rho=r$ and $t^1$ is from \eqref{wxwxx}.
\par
We claim that $m$ is a supersolution to equation \ref{oeq1} for any $t>0$ and $x\in\mathbb{R}$.
\par
To prove $m$ is a supersolution, we need to check that $m_t-m_{xx}-f(m)\ge0$ for all $t>0$ and $x\in\mathbb{R}$.  For $x<x_0(t+t^1)+t$ and $t>0$, since $m_t(t,x)=m_{xx}(t,x)=f(m(t,x))=0$, we have 
\begin{equation*}
    m_t(t,x)-m_{xx}(t,x)-f(m(t,x))=0.
\end{equation*} 
For $x\ge x_0(t+t^1)+t$ and $t>0$, by the definitions of $m$ and $w$ and Hypothesis \ref{fu}, we have
\begin{equation*}
    m_t(t,x)-f(m(t,x))\ge  r \frac{w(t+t^1,x-t)}{(1-\ln w(t+t^1,x-t))^\alpha}-w_x(t+t^1,x-t)-r \frac{w(t+t^1,x-t)}{(1-\ln w(t+t^1,x-t))^\alpha}=-w_x(t+t^1,x-t).
\end{equation*}
Therefore, for all $x\ge x_0(t+t^1)+t$ and $t>0$, since $x-t\ge x_0(t+t^1)>x_0(t)$ and $t+t^1>t^1$ for $t>0$, by Lemma \ref{lem22},we have 
\begin{equation*}
        m_t(t,x)-m_{xx}(t,x)-f(m(t,x))
        \ge -w_x(t+t^1,x-t) -w_{xx}(t+t^1,x-t)\ge 0.
\end{equation*}
\par
When $t=0$, since $u_0\le 1$ and $w(\cdot,x)$ is nondecreasing for each $x\in\mathbb{R}$, we have that $m(0,x)=1\ge u_0(x)=u(0,x)$ for  $x<x_0(t^1)$, and that $m(0,x)= w(t^1,x)\ge w(0,x)= u_0(x)=u(0,x)$ for  $x\ge x_0(t^1)$.
The comparison principle then yields that
\begin{equation}\label{lmm}
    u(t,x)\le m(t,x) \quad \text{for all }t>0 \text{ and }x\in\mathbb{R}.
\end{equation}
\begin{proof}[The proof of the upper bound]
   Now, we prove that, for any $\lambda\in(0,1)$ and large time $t$, there are a constant $\bar C_\lambda>0$ we have 
   \begin{equation}\label{lmub}
       E_\lambda(t)\subset u_0^{-1}\Big\{[\bar C_\lambda e^{-(r(\alpha+1)t)^\frac{1}{\alpha+1}},1]\Big\}.
   \end{equation}
  It follows from \eqref{lmm} that 
  \begin{equation*}
     u(t,x)\le m(t,x) \le w( t+t^1,x-t)\quad \text{for all }t>0 \text{ and }x\in\mathbb{R}.
  \end{equation*}
  If we pick $y\in E_\lambda(t)$, then $w(t+t^1,y-t)\ge\lambda$.
  Thus, by the definitions of $w$ and $\varphi_0$, we have 
  \begin{equation*}
        w(t+t^1,y-t)= \exp \Big\{1-[(1-\ln u_0(y-t))^{\alpha+1}-\rho(\alpha+1)(t+t^1)]^{\frac{1}{\alpha+1}} \Big\}\ge\lambda,
    \end{equation*}
    whence 
  \begin{equation*}
      u_0(y-t)\ge \exp\Big\{1-\left(r(\alpha+1)(t+t^1)+(1-\ln \lambda)^{\alpha+1}\right)^{\frac{1}{\alpha+1}}\Big\}.
  \end{equation*}
  
By the assumption $\varphi_0'(x)=o(\varphi_0^{-\alpha}(x))$ as $x\to +\infty$, we have, for some $C>0$,
\begin{equation*}
    \frac{U_0(Ce^{-(r(\alpha+1)t)^\frac{1}{\alpha+1}})}{t}\to +\infty \quad \text{as } t\to \infty,
\end{equation*}
where $U_0(z):=\sup \{x\in\mathbb{R},u_0(x)=z\}$.
Therefore, there is a time $\bar t_\lambda$ and a constant $\bar C_\lambda$ such that
\begin{equation}\label{lamee}
    E_\lambda(t)\subset u_0^{-1}\Big\{[\bar C_\lambda e^{-(r(\alpha+1)t)^\frac{1}{\alpha+1}},1]\Big\} \quad \forall t\ge \bar t_\lambda,
\end{equation}
proving \eqref{lmub}.

\end{proof}

\subsection{The lower bound}
Assuming additionally that $\varphi_0''(x)\le0$ for large $x$, we derive a more precise lower bound. To do so, we take $\rho=r$ and recall 
\begin{equation}\label{lvxt1}
	v(t,x)=\left\{
 \begin{aligned}
 &\frac{1}{4M},&x\le x_M(t),\\
 &g(w(t,x)),& x>x_M(t),
 \end{aligned}\right.
\end{equation}
where $ M\ge M_0=\max\left\{\frac{1}{2\zeta},\frac{1}{4s_0}\right\}$.
We claim that 
\begin{equation}\label{lbo}
	E_\lambda(t)\subset u_0^{-1}\left\{\left(0,\  \underline Ce^{-[r(\alpha+1)t]^{\frac{1}{\alpha+1}}}\right]\right\}\quad \text{for $t$ large enough}.
\end{equation}
In view of \eqref{wxx}, since $\varphi_0''(x)\le 0$ for large $x$, we obtain $w_{xx}\ge 0$ for any $x\ge x_M(t)$ with large enough $M$, say $M>M^*\ge M_0$.
For $x>x_M(t)$, it follows from the assumption $\varphi_0'(x)=o(\varphi_0^{-\alpha}(x))$ as $x\to +\infty$, the definition \eqref{lvxt1} of $v$ and \eqref{wx} that we have for $M>M^*$, up to enlarge $M^*$,
\begin{equation*}
     \begin{aligned}
        v_{xx}(t,x)&=(1-2M w)w_{xx}-2 M w_x^2\\
         &\ge -2 M \left(\varphi_0'(1+\varphi_0)^{\alpha}\right)^2\frac{w^2}{(1-\ln w)^{2\alpha}}\\
         &\ge -\frac{1}{4} rM\frac{w^2}{(1-\ln w)^{2\alpha}}.
     \end{aligned}
\end{equation*}
On the other hand, for $x\le x_M(t)$, we have $v_{xx}(t,x)=0.$
\par
In view of \eqref{llfu1} and \eqref{llfu2}, similar to \eqref{laamtmxxfm1}, we obtain, for $x\le x_M(t)$, 
\begin{equation*}
    (v_t-v_{xx}-f(v))(t,x)
    \le 0,
\end{equation*}
whereas, for $x>x_M(t)$, we have
\begin{equation*}
    (v_t-v_{xx}-f(v))(t,x)\le\frac{w(t,x)}{(1-\ln w(t,x))^\alpha}\Big((-\frac{3}{4}rM+rK)w(t,x)
    +\alpha \tilde c M\frac{w(t,x)}{1-\ln w(t,x)}+\alpha \tilde c KM^2\frac{w^{3}}{1-\ln w(t,x)}\Big).
\end{equation*}
For $0<w\le \frac{1}{2M}$, when $M\ge  M^{**}:=\max\left\{\frac{1}{2}\exp\left(4\alpha \tilde c(1+\frac{1}{r})-1\right),\frac{K}{4}\right\}$ , we have
\begin{equation*}
    \frac{r+KMw^2}{1-\ln w}\le \frac{r+1}{1+\ln(2M)}\le \frac{r}{4\alpha \tilde c}.
\end{equation*}
Thus, for any $x>x_M(t)$, we take $M$ large enough, say
\begin{equation*}
    M\ge \bar M:=\max\Big\{2rK,M^*, M^{**}\Big\},
\end{equation*}
so that 
\begin{equation*}
    r\alpha \tilde c  M\frac{w(t,x)}{1-\ln w(t,x)}+\alpha \tilde c K M^2 \frac{w^{3}(t,x)}{1-\ln w(t,x)}=\alpha \tilde c  M\frac{r+KMw^2(t,x)}{1-\ln w(t,x)} w(t,x)\le \frac{1}{4}rMw (t,x),
\end{equation*}
whence
\begin{equation*}
    (v_t-v_{xx}-f(v))(t,x)
    \le\frac{w^{2}(t,x)}{(1-\ln w(t,x))^\alpha}\left(-\frac{1}{2}M+rK\right)\le 0.
\end{equation*}

\par In view of \eqref{lvxt1}, we notice that:
\begin{itemize}
    \item when $x>x_M(0)$, we have $ v(0,x)=g(w(0,x))\le w(0,x)=u_0(x)$;
    \item when $\xi_1\le x\le x_M(0)$, since $u_0$ is nonincreasing on $[\xi_1,+\infty)$, we have $v(0,x)=\frac{1}{4M}<\frac{1}{2M}\le u_0(x)$;
    \item when $x< \xi_1$, since $M>\frac{1}{2\zeta}$, we have $v(0,x)=\frac{1}{4M}\le\frac{1}{2M}<\zeta\le u_0(x)$.
\end{itemize}
 Thus, we obtain $v(0,x)\le u_0(x)=u(0,x)$ for all $x\in\mathbb{R}$.
Therefore, the comparison principle implies that 
\begin{equation*}
    v(t,x)\le u(t,x) \quad \text{for all }x\in\mathbb{R},\ t>0.
\end{equation*}
\par
It follows from the proof of the lower bound of Lemma \ref{gen} that, for $\lambda\in (0,\Theta)$, if $x\in E_\lambda(t)$, then there exists a time $\underline t'_\lambda>t'_\theta$ such that 
\begin{equation}\label{lbe1}
	u_0(x)\le u_0(y_\theta(t))=\exp\left\{1-\left[r(\alpha+1)t+(1-\ln \theta)^{\alpha+1}\right]^{\frac{1}{\alpha+1}}\right\}\le \underline C_1 e^{-[r(\alpha+1)t]^{\frac{1}{\alpha+1}}} \quad \text{for any }t>\underline t'_\lambda,
\end{equation}
and, for $\lambda\in (\Theta,1)$, there exists a time $\underline t_\lambda>\max\{t'_\lambda, \underline t''_\lambda\}$ such that 
\begin{equation}\label{lbe2}
	u_0(x)<u_0(y_\theta(t-\underline t''_\lambda))=\exp\left\{1-\left[r(\alpha+1)(t-\underline t''_\lambda)+(1-\ln \lambda)^{\alpha+1}\right]^{\frac{1}{\alpha+1}}\right\}< \underline C_2e^{-[r(\alpha+1)t]^{\frac{1}{\alpha+1}}}
\end{equation}
for any $t>\underline t_\lambda$ which gives \eqref{lbo}.
Let $T_\lambda=\max\{\underline t_\lambda,\bar t_\lambda\}$ and $\underline C_\lambda=\max\{\underline C_1,\underline C_2\}$. Thus, combining \eqref{lamee}  \eqref{lbe1} and \eqref{lbe2}, the proof of Lemma \ref{gen1} is complete.

\section{Numerical simulations}\label{s5}
In this section, we provide some numerical simulations to illustrate the previous results. 
\par
To get an approximate solution for equation \eqref{oeq1}, we discretize the equation in space by the finite difference method and then use \textit{Implicit-Explicit} scheme (IMEX) \cite{ascher1997implicit,ascher1995implicit} to integrate it in time,
 where the implicit scheme handles the diffusion term while the explicit handles the reaction term. 
The influence of the initial data $u_0$ on the propagation speed is illustrated under the some fixed $\alpha$ in Figure \ref{fig:exp0.2}-\ref{fig:alge0.6}.
We mainly consider the initial data with two kinds of decay: sub-exponential decay and algebraic decay. In the following simulations, we all take $f(u)=\frac{u}{(1-\ln u)^\alpha}(1-u)$ for $u\in(0,1)$.
For the initial data with sub-exponential decay, we take the initial data to be $u_0=\min\{e^{-5x^\beta},1\}$ and $\alpha=0.2,\ 0.4,\ 0.6$. 
Figure \ref{fig:exp0.2}, Figure \ref{fig:exp0.4} and Figure \ref{fig:exp0.6} show that the acceleration can be observed over a small time range when $\beta$ is small. 
This is consistent with our theoretical results, that is, $x_\lambda(t)\asymp t^{\frac{1}{\beta(1+\alpha)}}$ tends to infinite as $\beta\to 0^+$. 
When $\beta$ is large enough, as we show in Theorem \ref{lin}, for initial data $u_0\lesssim e^{-\mu x^\beta}$ with $\beta>\frac{1}{\alpha+1}$ and $\mu>0$, the solution propagates at a finite rate.
Notice that the width of the solution becomes larger and larger as $\beta$ gets smaller and smaller. This is because the flattening effect \cite{bouin2021sharp,garnier2015transition}.

\begin{figure} 
\centering 
    \begin{subfigure}[b]{0.48\textwidth} 
        \includegraphics[width=\textwidth]{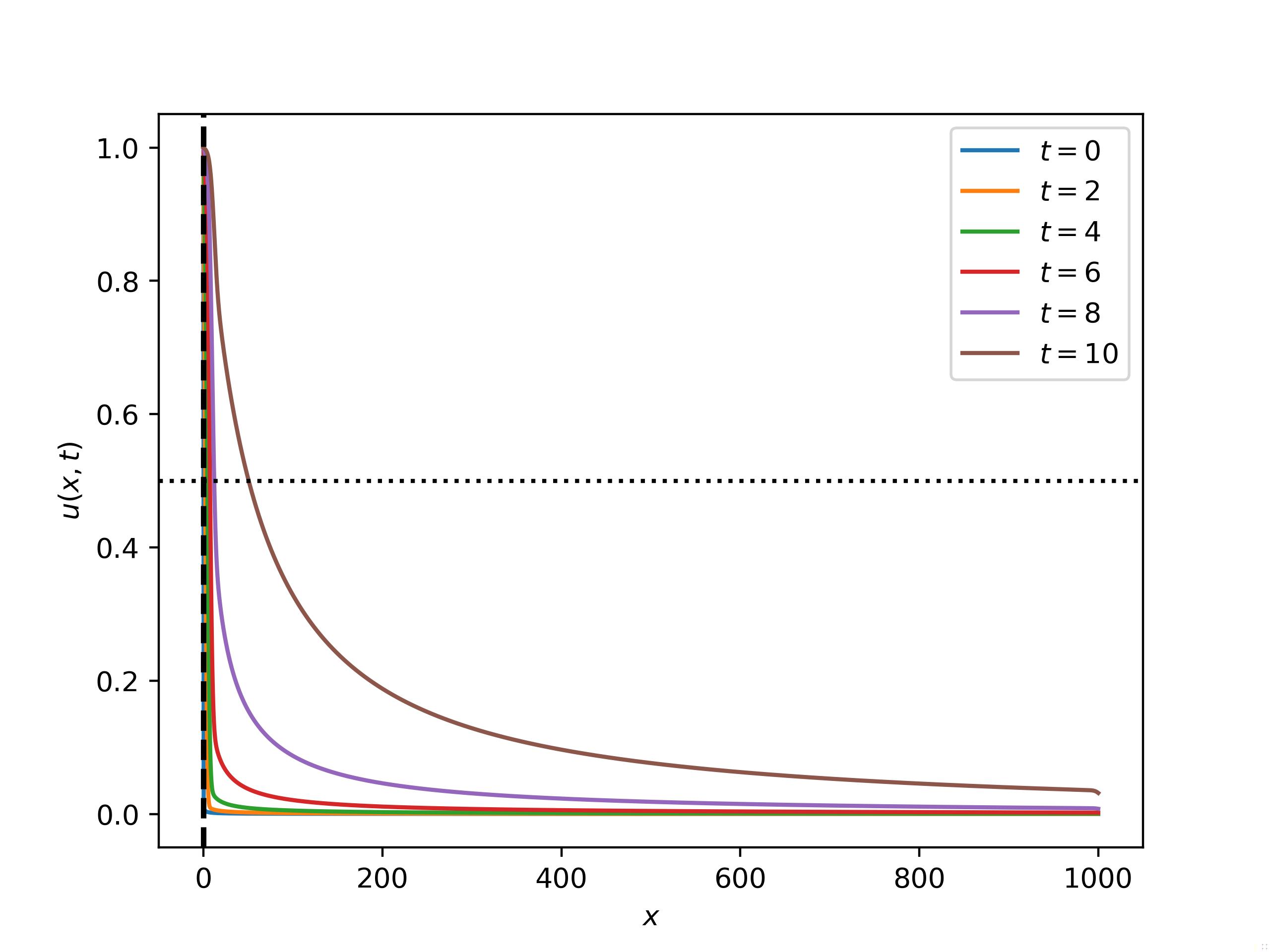} \caption{$\beta=0.1$} 
        \label{fig:exp0.2beta0.2} 
    \end{subfigure} \quad
    \begin{subfigure}[b]{0.48\textwidth}  
        \includegraphics[width=\textwidth]{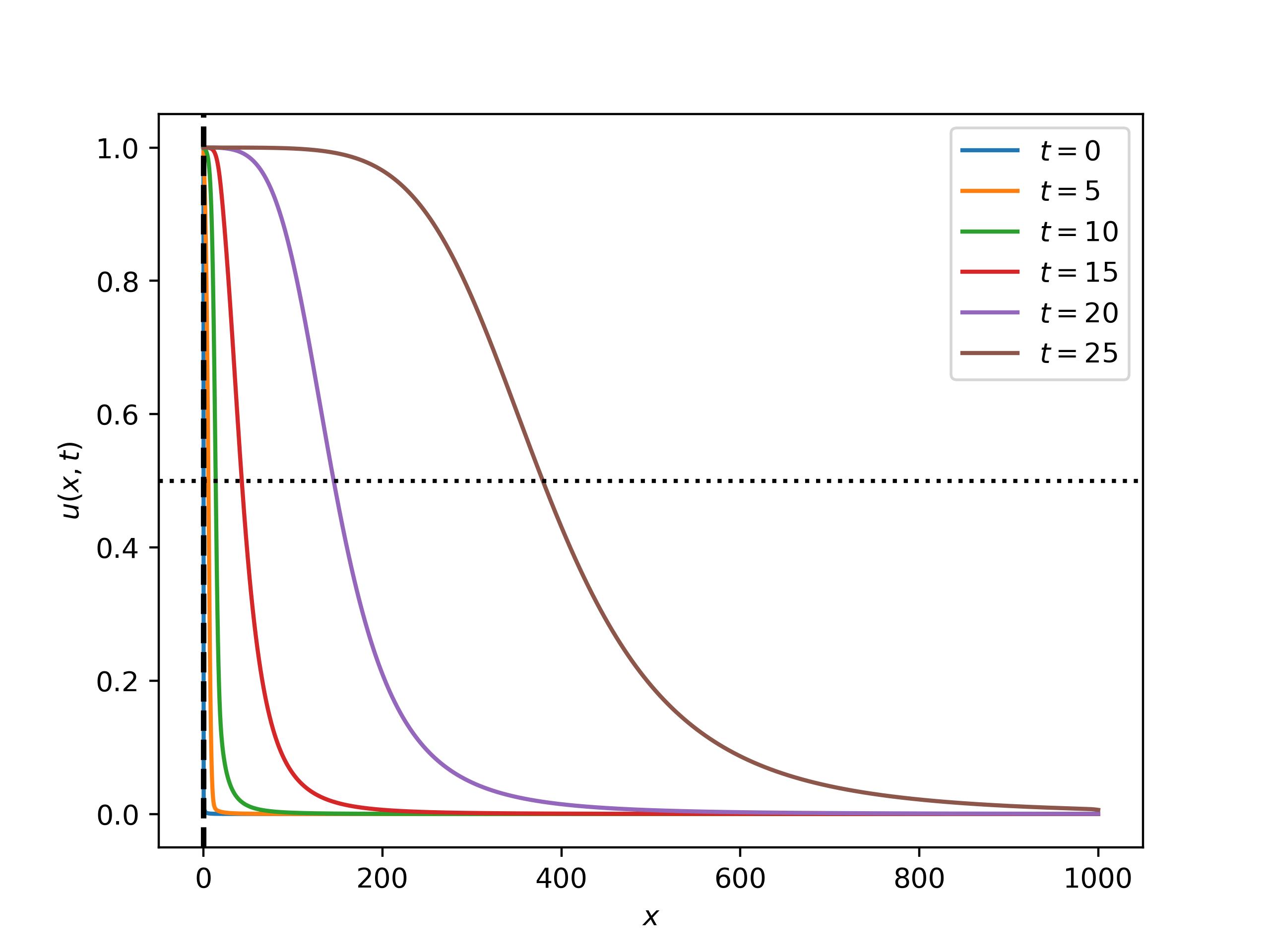} 
        \caption{$\beta=0.2$} 
    \end{subfigure} 
    
    \begin{subfigure}[b]{0.48\textwidth}  
        \includegraphics[width=\textwidth]{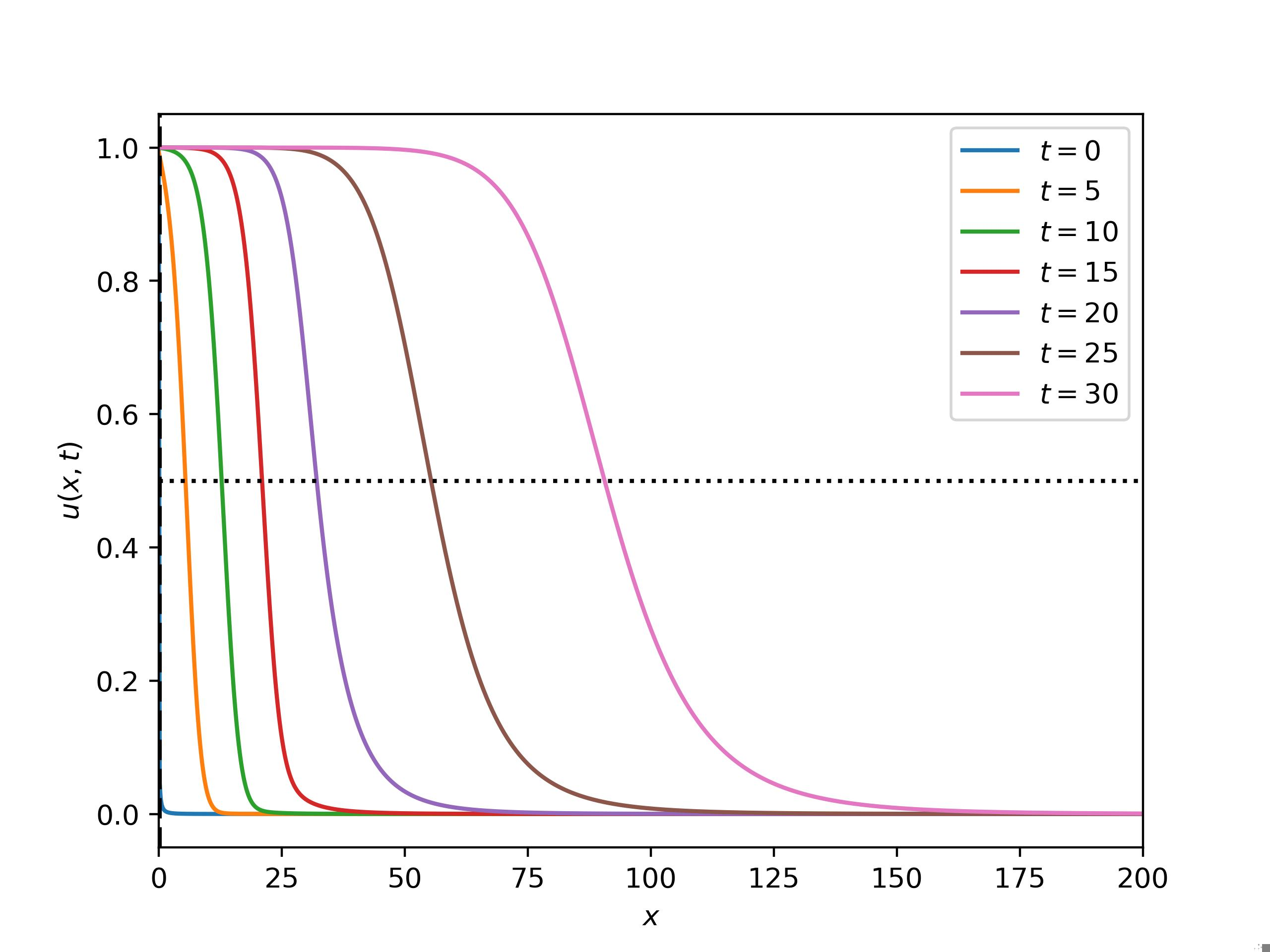} \caption{$\beta=0.3$} 
    \end{subfigure}
    \quad
    \begin{subfigure}[b]{0.48\textwidth}  
        \includegraphics[width=\textwidth]{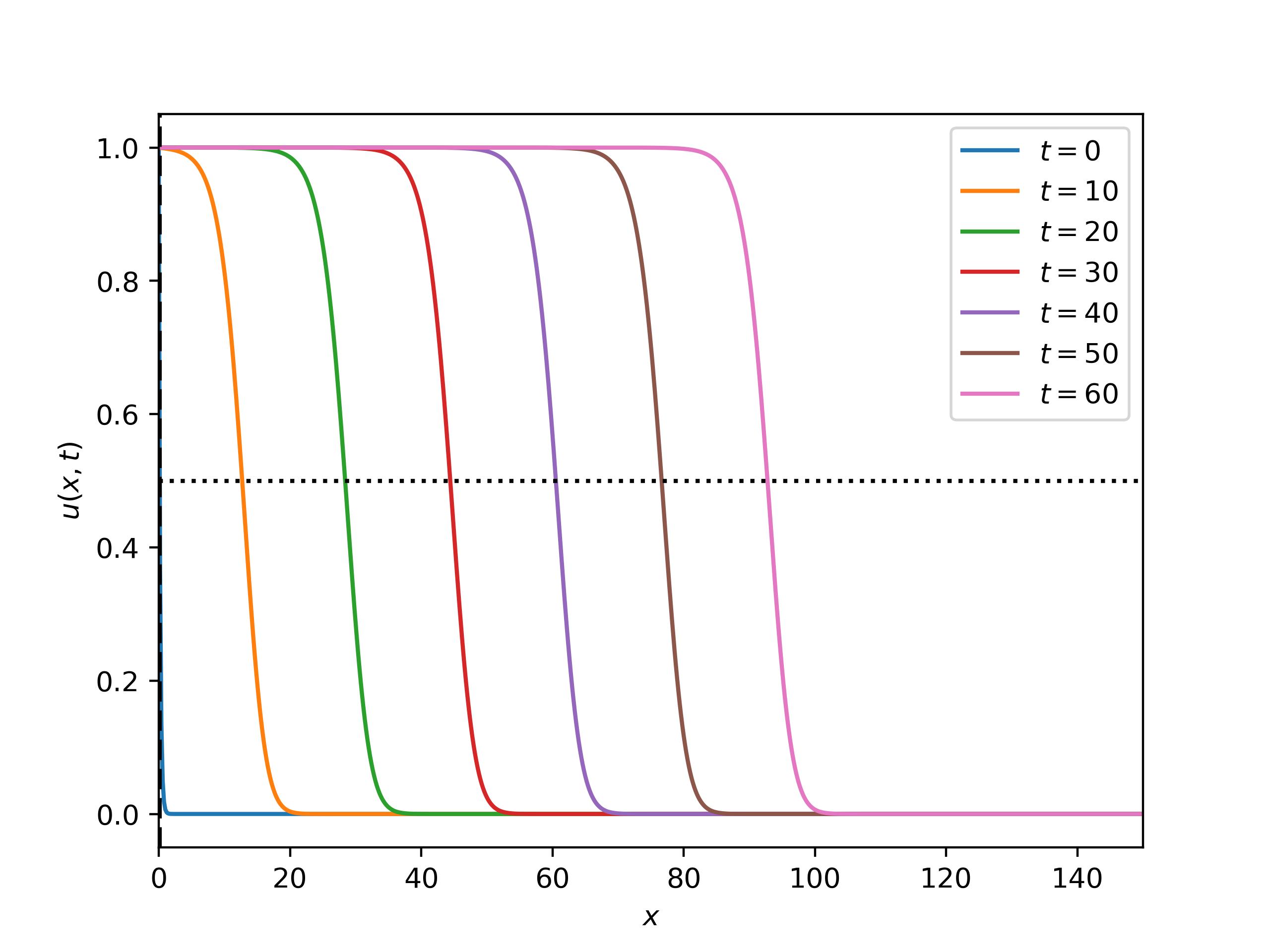} \caption{$\beta=1.0$} 
    \end{subfigure}
       
\caption{Numerical approximations of the solution to \eqref{oeq1} with the initial data $u_0(x)=\min\{e^{-5x^\beta},1\}$ at different times for $\alpha=0.2$ and different values of $\beta$. The threshold for acceleration is $\beta=\frac{1}{\alpha+1}=\frac{5}{6}$.}
\label{fig:exp0.2}
\end{figure} 

\begin{figure} 
\centering 
    \begin{subfigure}[b]{0.48\textwidth} 
        \includegraphics[width=\textwidth]{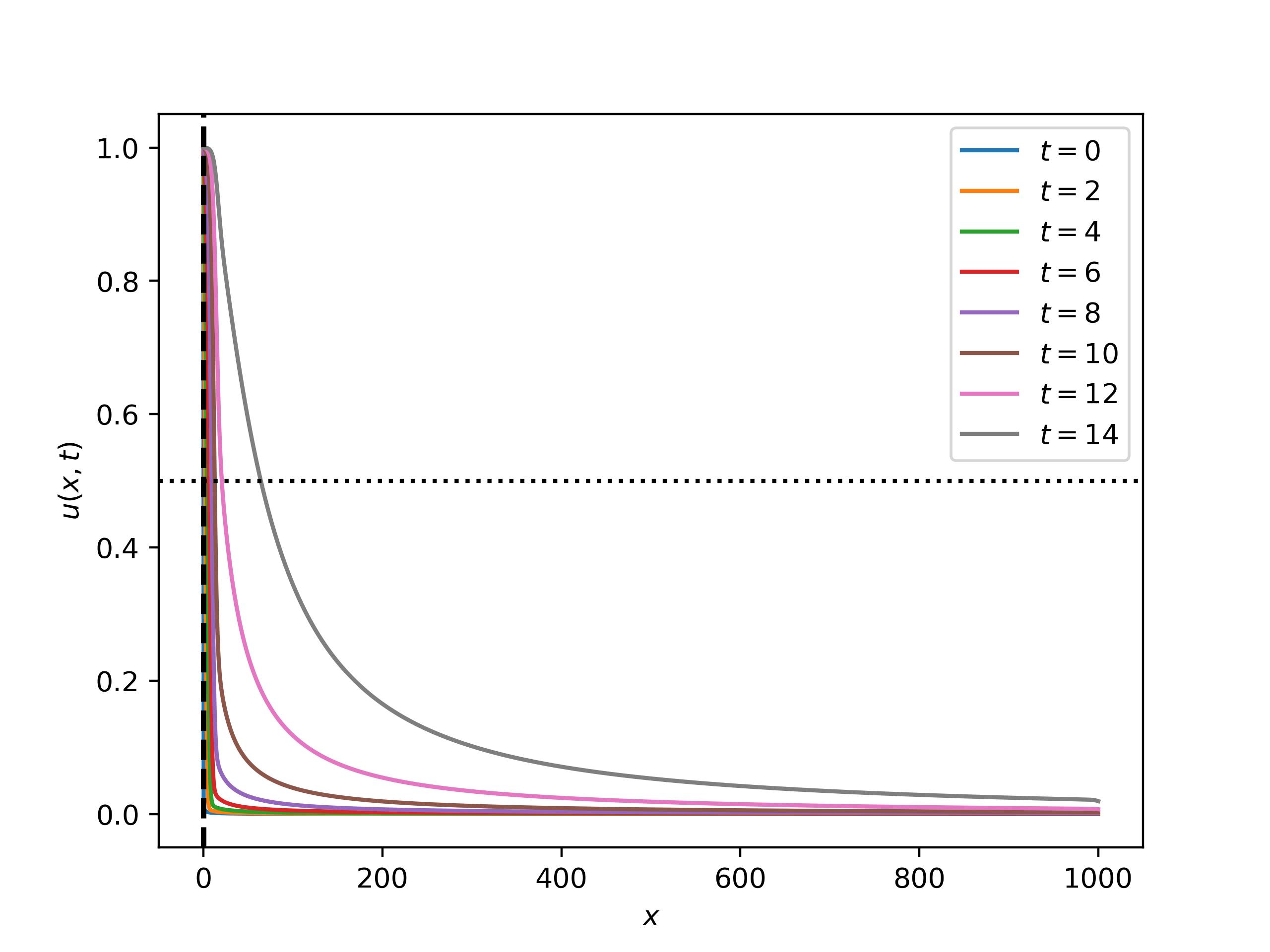} \caption{$\beta=0.1$} 
    \end{subfigure} 
    \quad
    \begin{subfigure}[b]{0.48\textwidth}  
        \includegraphics[width=\textwidth]{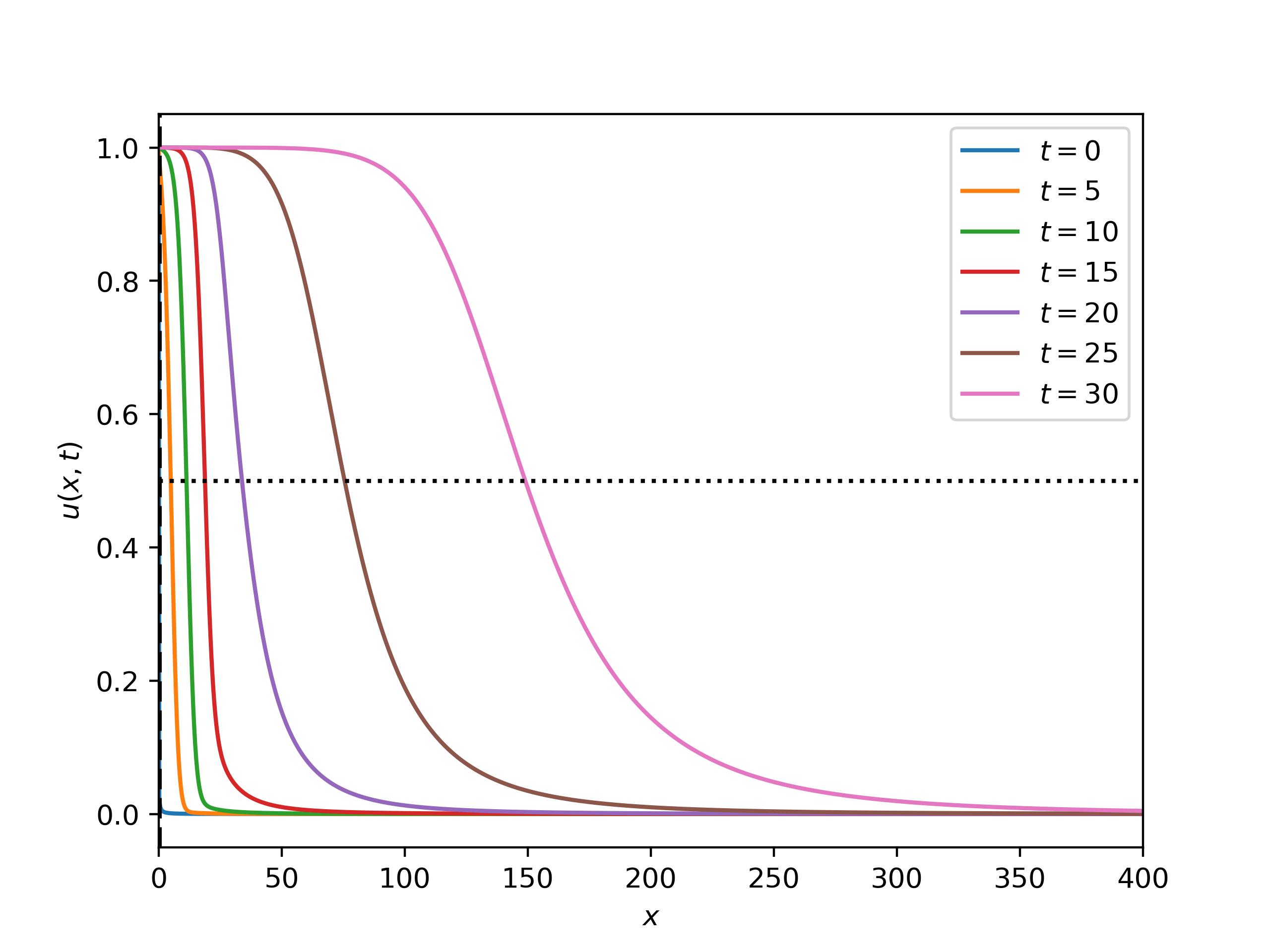} 
        \caption{$\beta=0.2$} 
    \end{subfigure} 
    
    \begin{subfigure}[b]{0.48\textwidth}  
        \includegraphics[width=\textwidth]{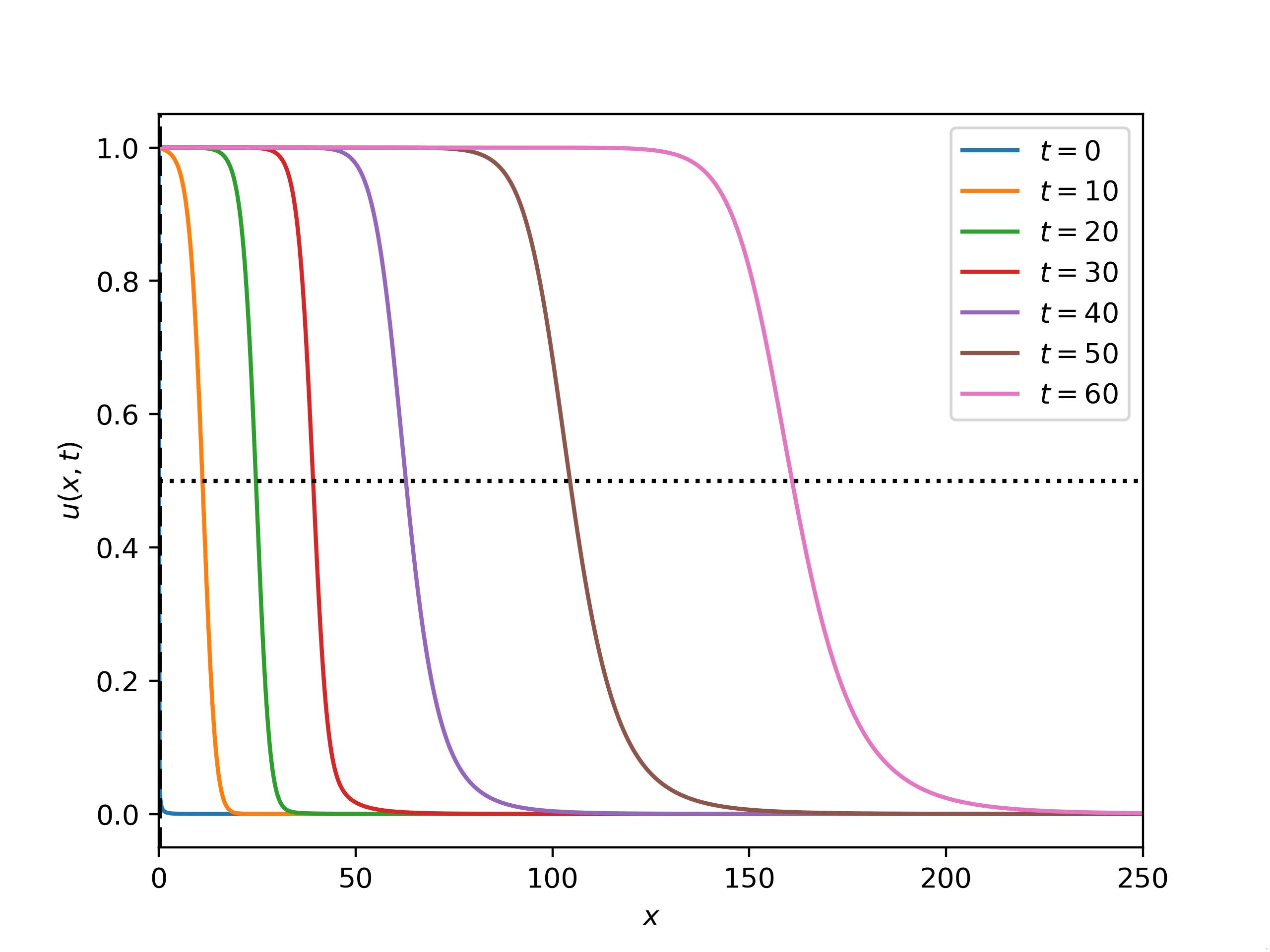} \caption{$\beta=0.3$} 
    \end{subfigure} 
    \quad
    \begin{subfigure}[b]{0.48\textwidth}  
        \includegraphics[width=\textwidth]{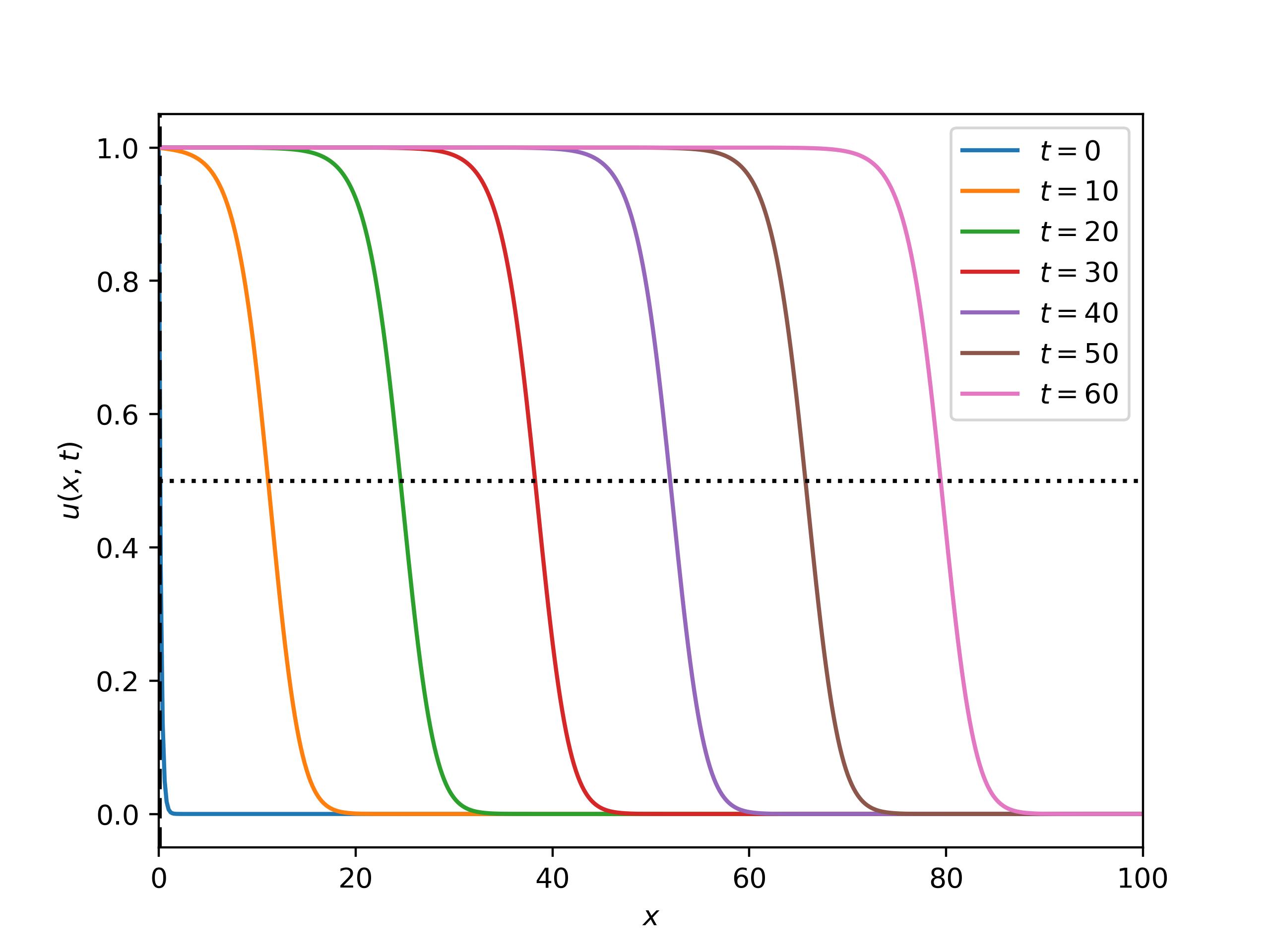} \caption{$\beta=1.0$} 
    \end{subfigure} 
   
\caption{Numerical approximations of the solution to \eqref{oeq1} with the initial data $u_0(x)=\min\{e^{-5x^\beta},1\}$ at different times for $\alpha=0.4$ and different values of $\beta$. The threshold for acceleration is $\beta=\frac{1}{\alpha+1}=\frac{5}{7}$.}
\label{fig:exp0.4}
\end{figure} 
\begin{figure} 
\centering 
    \begin{subfigure}[b]{0.48\textwidth} 
        \includegraphics[width=\textwidth]{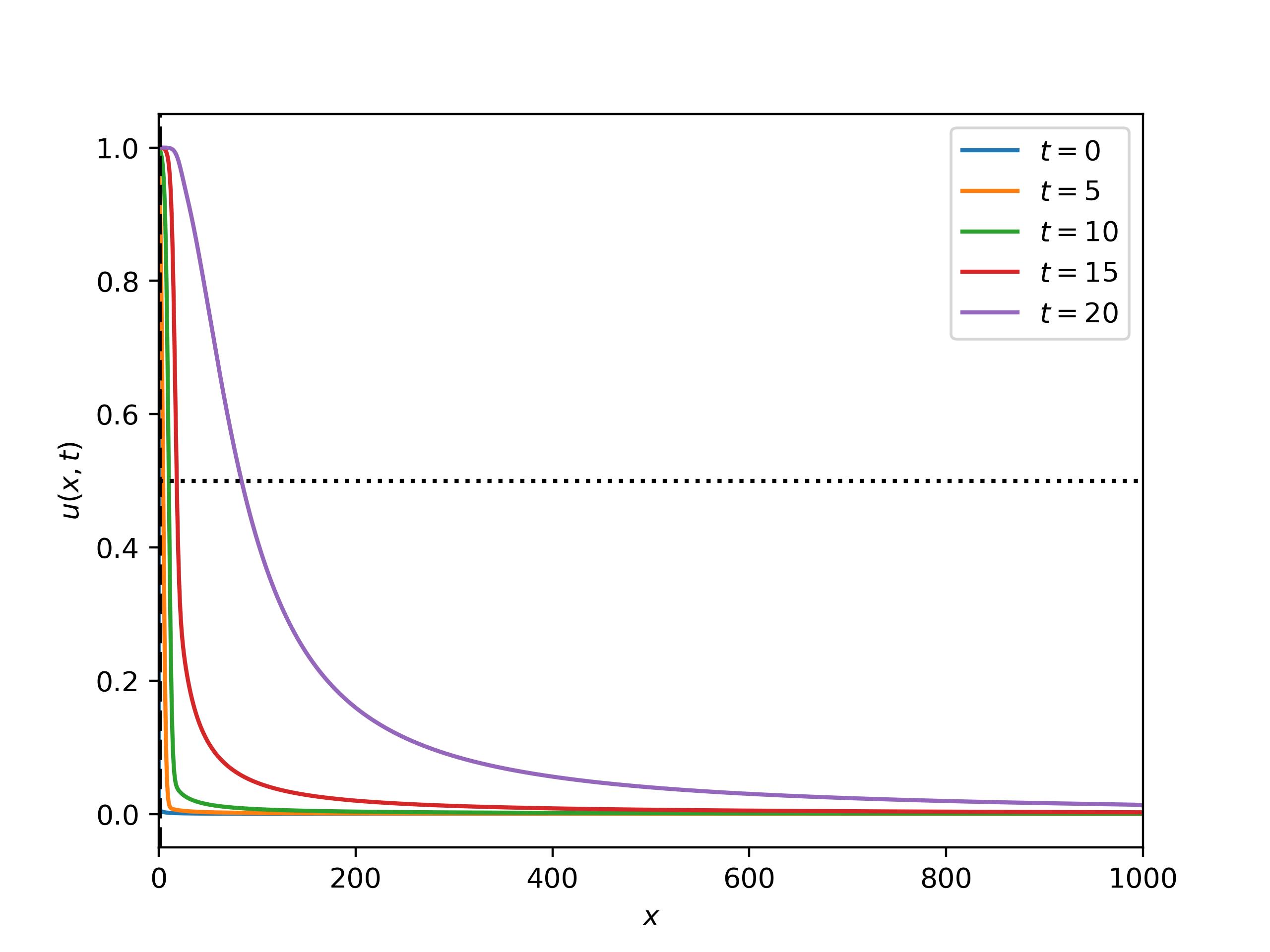} \caption{$\beta=0.1$} 
    \end{subfigure} 
    \quad
    \begin{subfigure}[b]{0.48\textwidth}  
        \includegraphics[width=\textwidth]{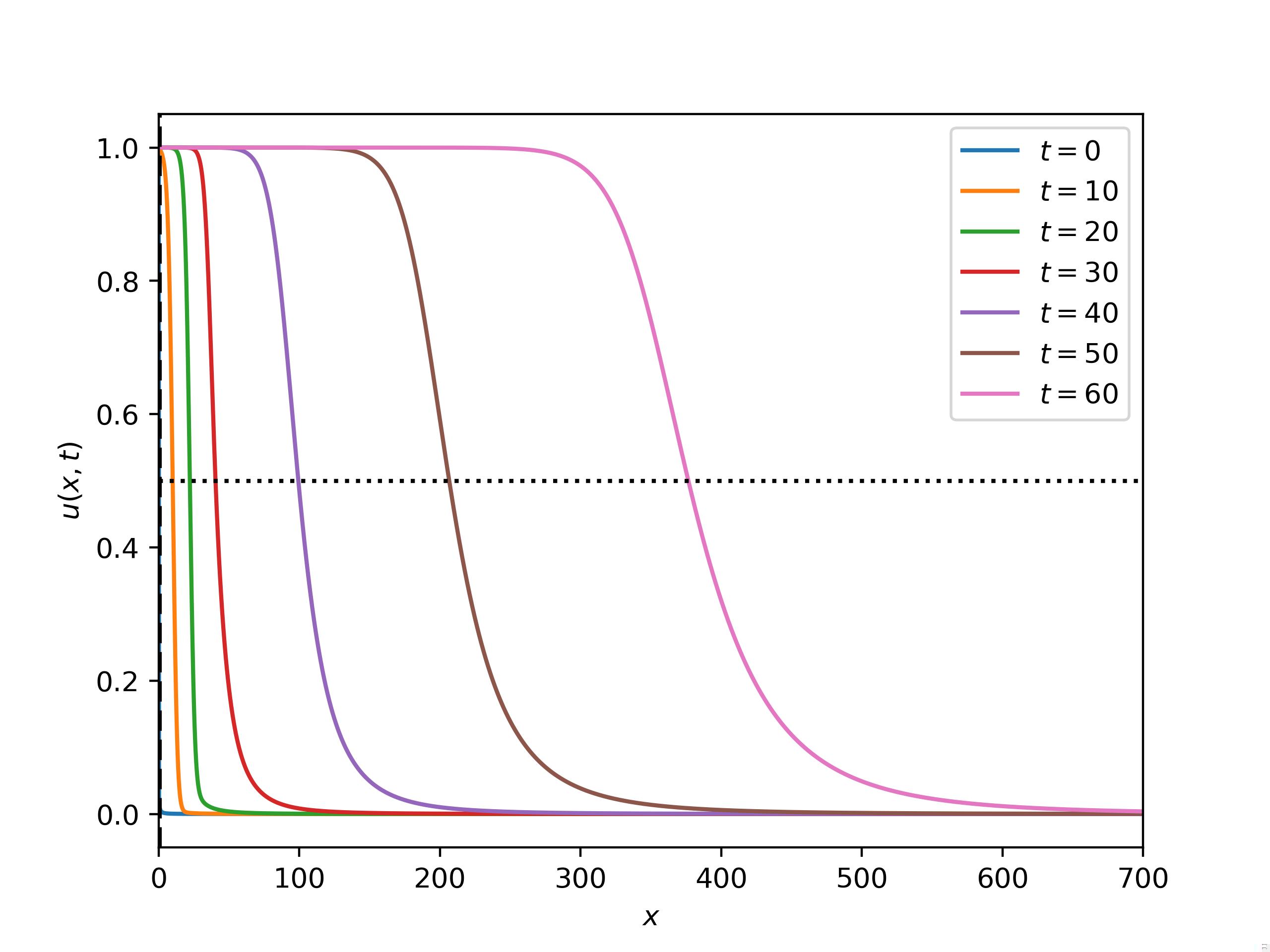}
        \caption{$\beta=0.2$} 
    \end{subfigure} 
    
    \begin{subfigure}[b]{0.48\textwidth}  
        \includegraphics[width=\textwidth]{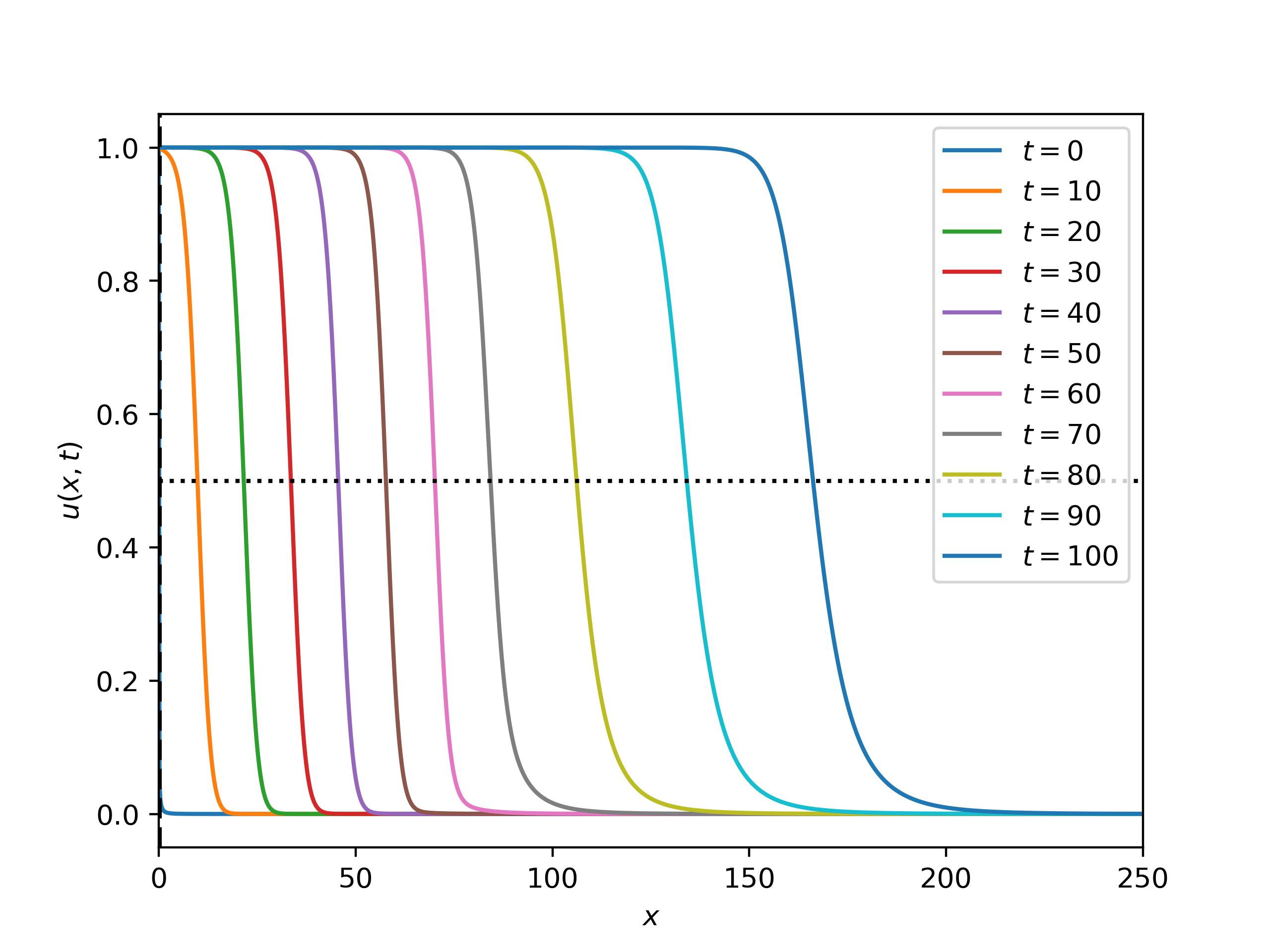} \caption{$\beta=0.3$} 
    \end{subfigure} 
     \quad
    \begin{subfigure}[b]{0.48\textwidth}  
        \includegraphics[width=\textwidth]{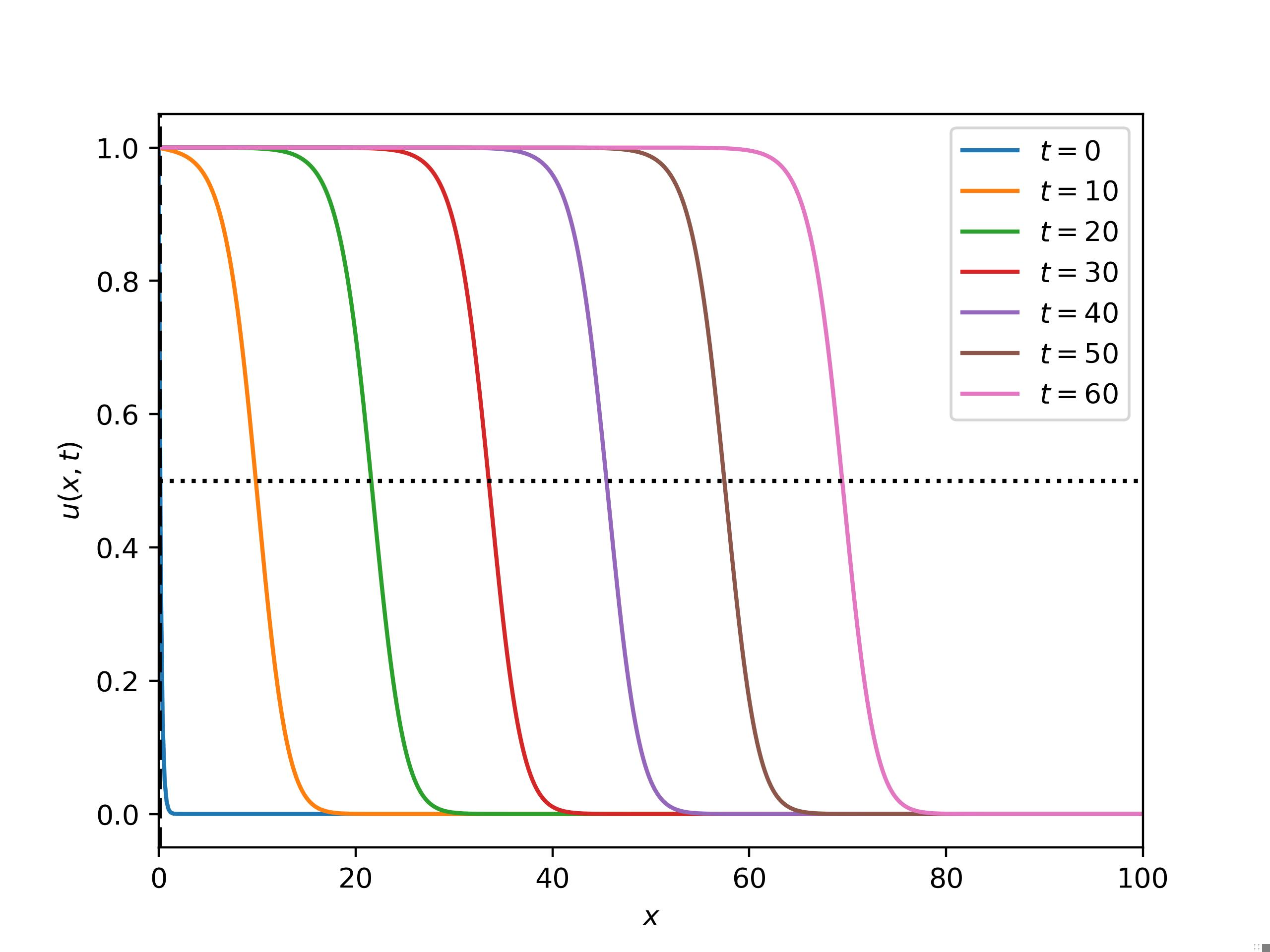} \caption{$\beta=1.0$} 
    \end{subfigure} 
\caption{Numerical approximations of the solution to \eqref{oeq1} with the initial data $u_0(x)=\min\{e^{-5x^\beta},1\}$ at different times for $\alpha=0.6$ and different values of $\beta$. The threshold for acceleration is $\beta=\frac{1}{\alpha+1}=\frac{5}{8}$. }
\label{fig:exp0.6}
\end{figure} 

\par
For initial data with algebraic decay, we take $u_0=\min\{\frac{1}{1+100x^\beta},1\}$ and $\alpha=0.2,\ 0.4,\ 0.6$. In Figure \ref{fig:alge0.2}, Figure \ref{fig:alge0.4}, and Figure \ref{fig:alge0.6}, we observe that decreasing the parameter $\beta$ leads to an increase of the propagation speed. Our theoretical findings support this observation, as demonstrated by the fact that $x_\lambda(t)\asymp \exp{\tfrac{(r(\alpha+1)t)^{\tfrac{1}{\alpha+1}}}{\beta}}$ tends to infinity as $\beta\to 0^+$. We can also observe the flattening effect. Therefore, the decay of the initial data is the key to the propagation of solution to equation \eqref{oeq1}.  When the initial data increases, meaning $\beta$ decreases, the propagation speed also increases. 
\par In Figure \ref{fig:com_level_sets}, we provide a comparison between the largest element $x_\lambda(t)$ of level sets $E_\lambda(t)$ of the solution with three different types of initial data. Observe that the slope of curve for the algebraic decay case is maximum, followed by the sub-exponential decay, and the sub-exponentially bounded case show a straight line. This is consistent with our theoretical results. 
We fit the corresponding theoretical results for each cases, as shown in the thick continuous curves in the figure \ref{fig:com_level_sets}. Notice that in each pair of curves when time $t$ is large enough, our experimental results are consistent with the theoretical results. The curves we choose with theoretical rates are $x_\lambda(t)=1.9t-4.0$, $x_\lambda(t)=0.0013t^{\frac{1}{0.28}}+40.0$ and $x_\lambda(t)= 0.0236e^{(1.4t)^{\frac{1}{1.4}}}+10$ respectively. Here, in order to better observe the trend of each pair of curves, we make a small downward translation for the thick continuous curves.
\begin{figure} 
\centering 
    \begin{subfigure}[b]{0.31\textwidth} 
        \includegraphics[width=\textwidth]{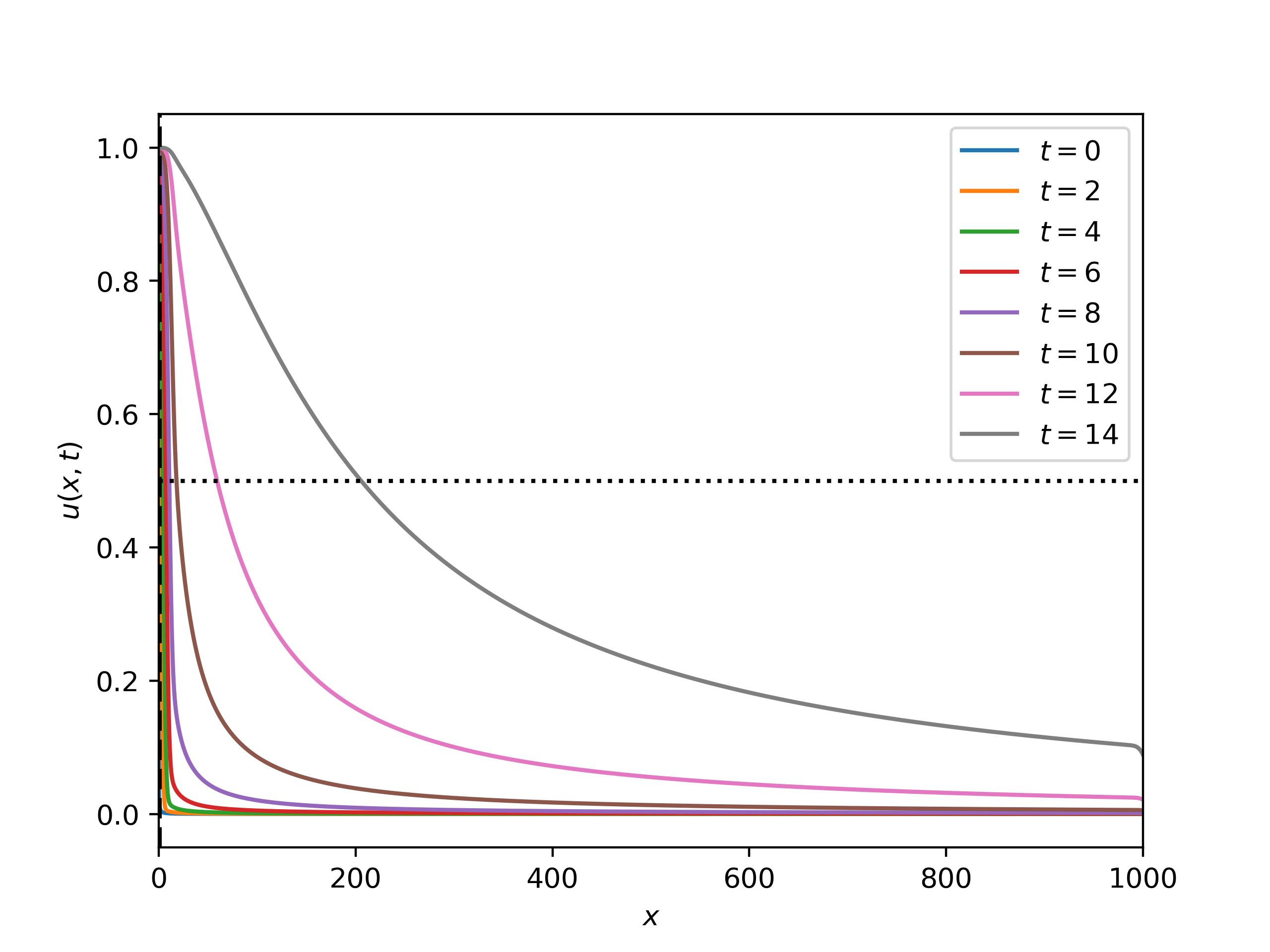} \caption{$\beta=1.0$} 
    \end{subfigure} 
    \quad
    \begin{subfigure}[b]{0.31\textwidth}  
        \includegraphics[width=\textwidth]{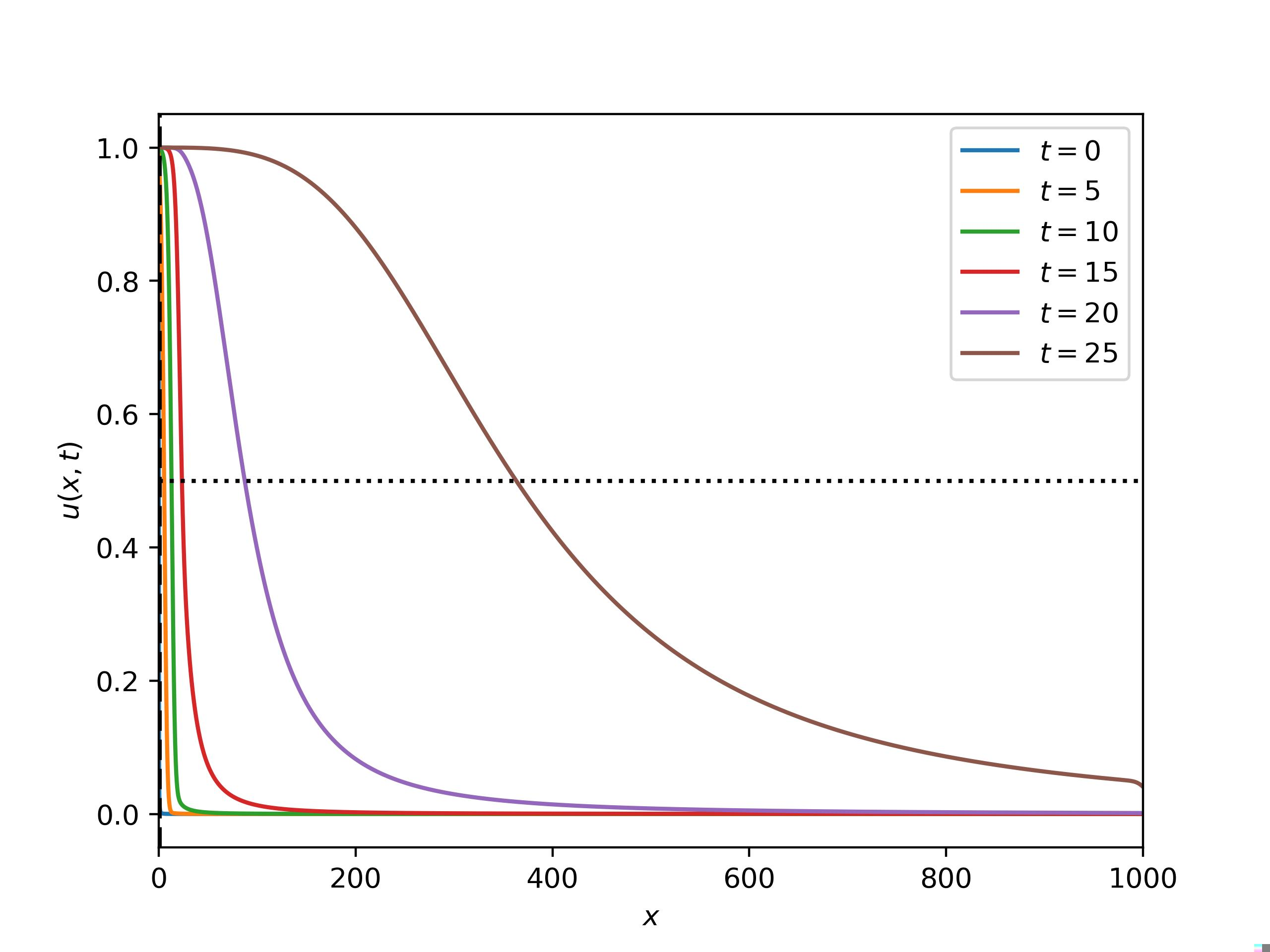} 
        \caption{$\beta=2.0$} 
    \end{subfigure} 
    \quad 
    \begin{subfigure}[b]{0.31\textwidth}  
        \includegraphics[width=\textwidth]{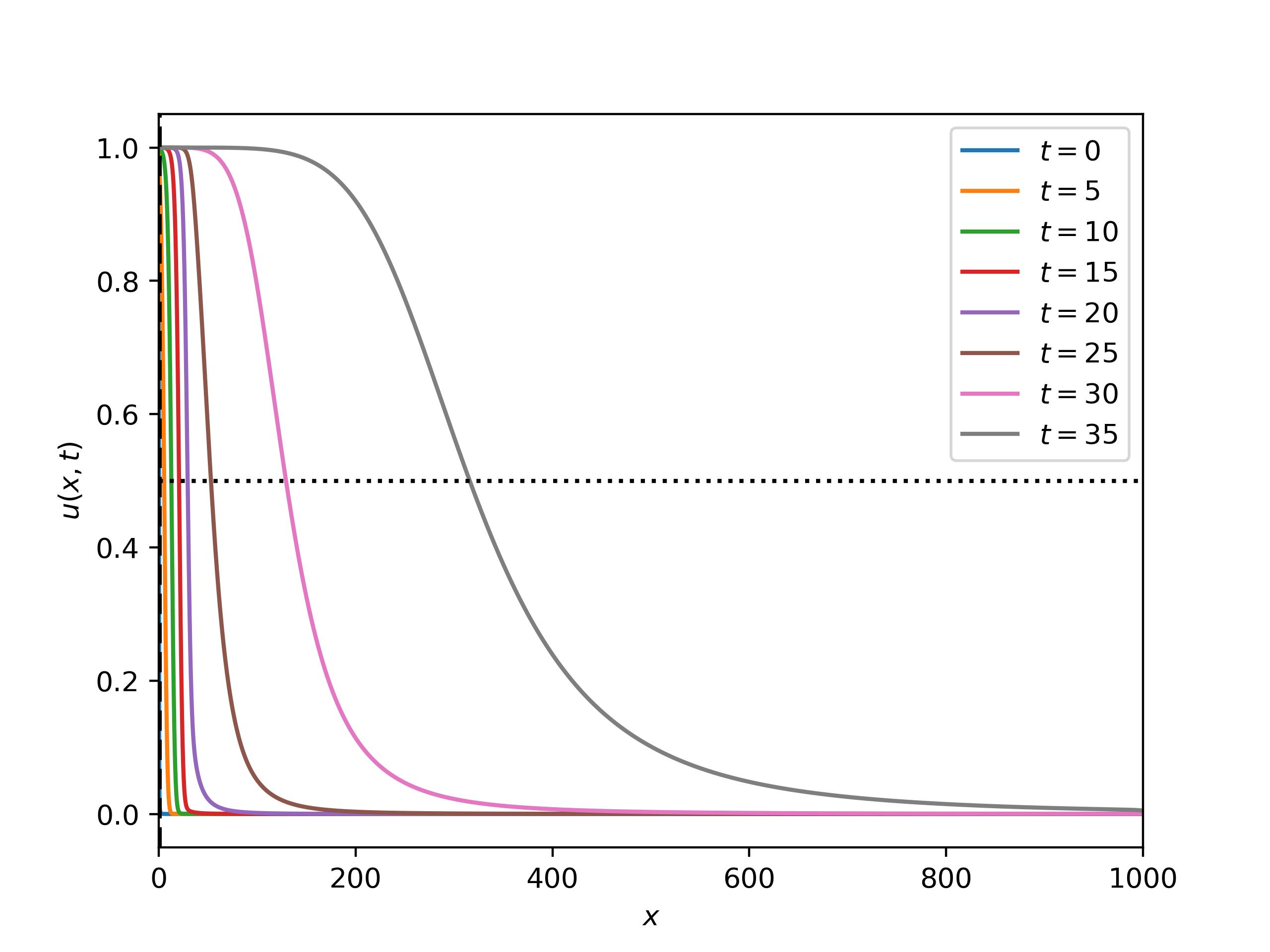} \caption{$\beta=3.0$} 
    \end{subfigure} 
\caption{Numerical approximations of the solution to \eqref{oeq1} with the initial data $u_0(x)=\min\{\frac{1}{1+100x^\beta},1\}$ at different times for $\alpha=0.2$ and different values of $\beta$.}
\label{fig:alge0.2}
\end{figure}
\begin{figure} 
\centering 
    \begin{subfigure}[b]{0.31\textwidth} 
        \includegraphics[width=\textwidth]{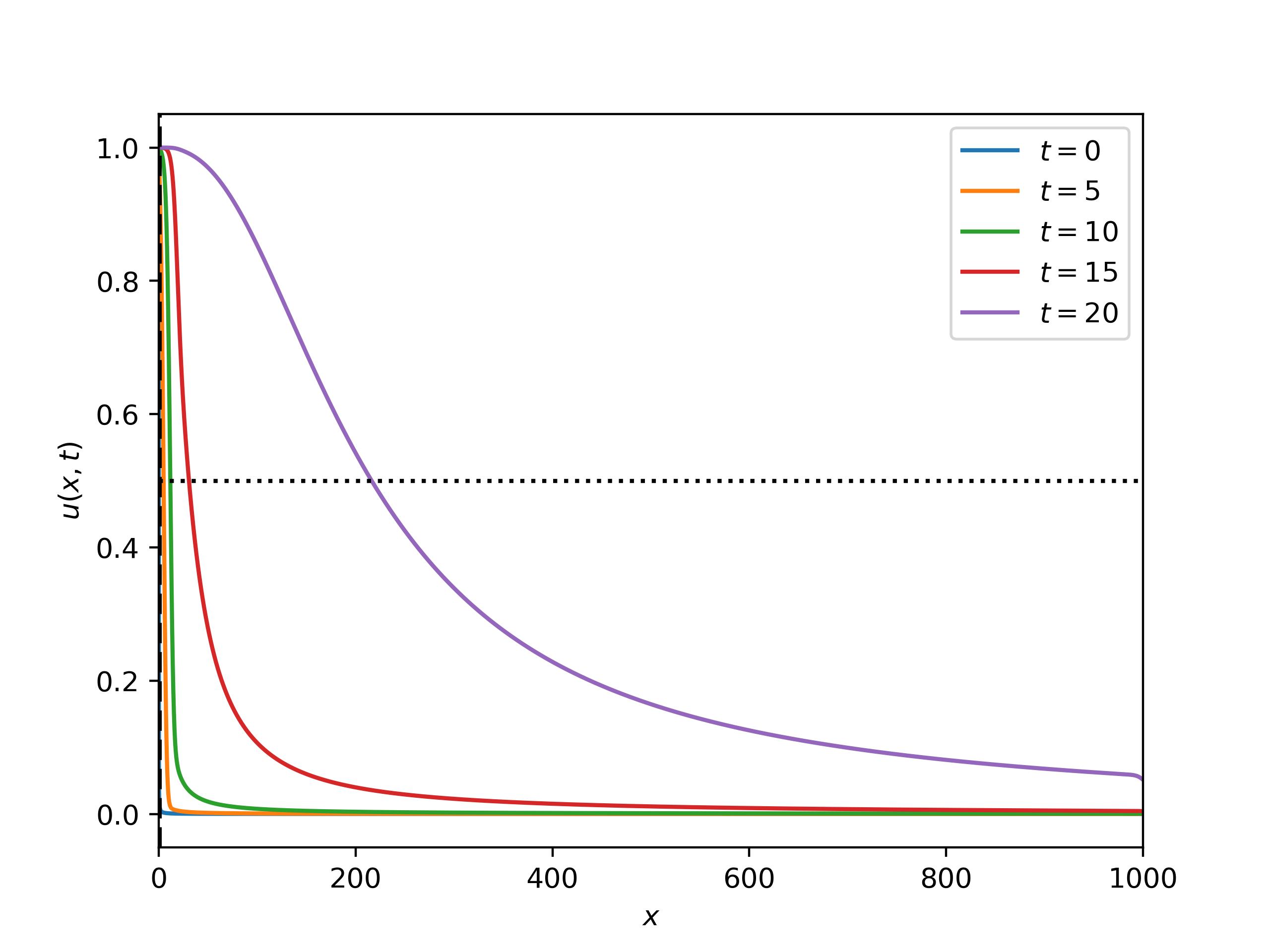} \caption{$\beta=1.0$} 
    \end{subfigure} 
    \quad
    \begin{subfigure}[b]{0.31\textwidth}  
        \includegraphics[width=\textwidth]{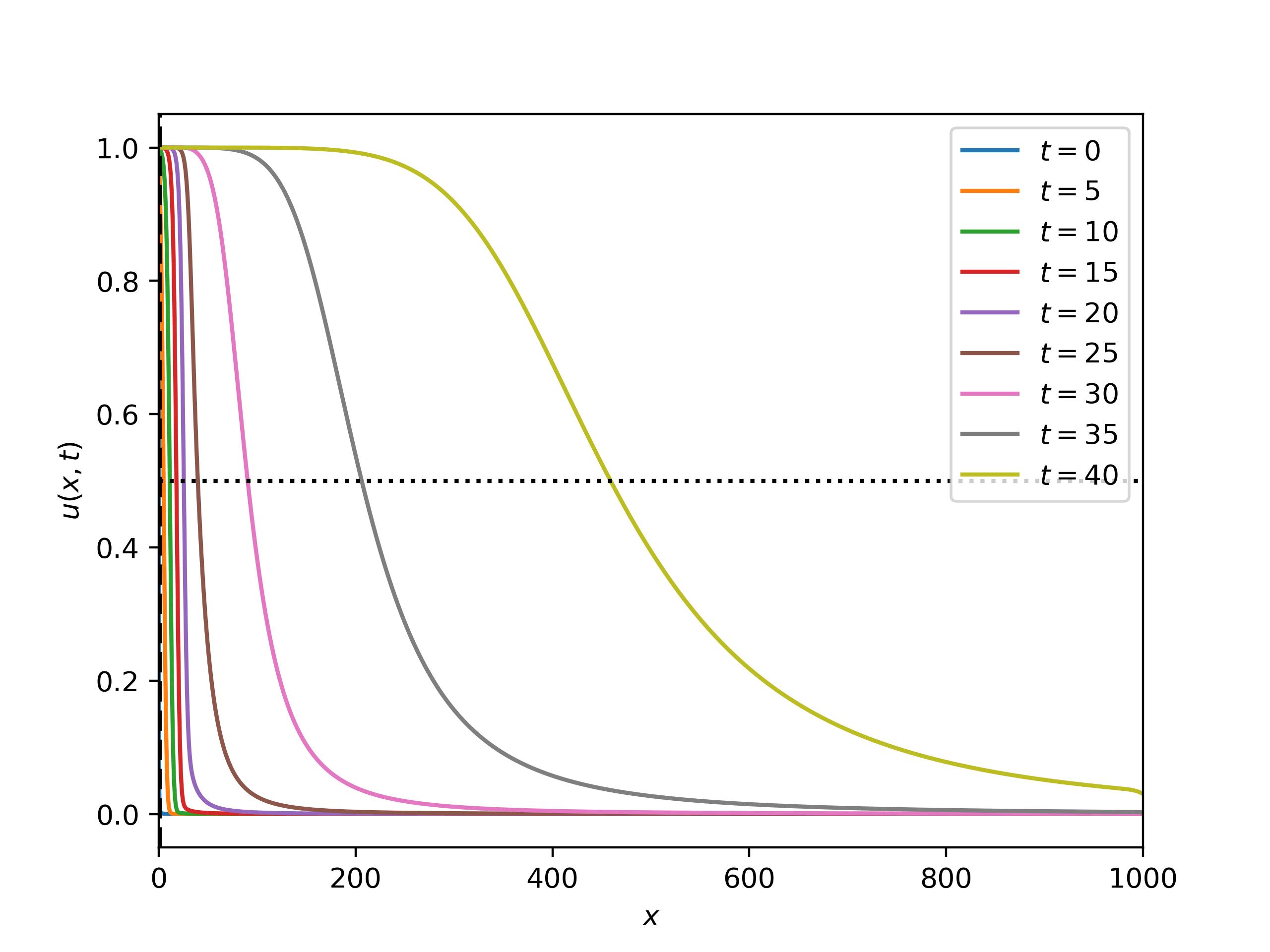} 
        \caption{$\beta=2.0$} 
    \end{subfigure} 
    \quad
    \begin{subfigure}[b]{0.31\textwidth}  
        \includegraphics[width=\textwidth]{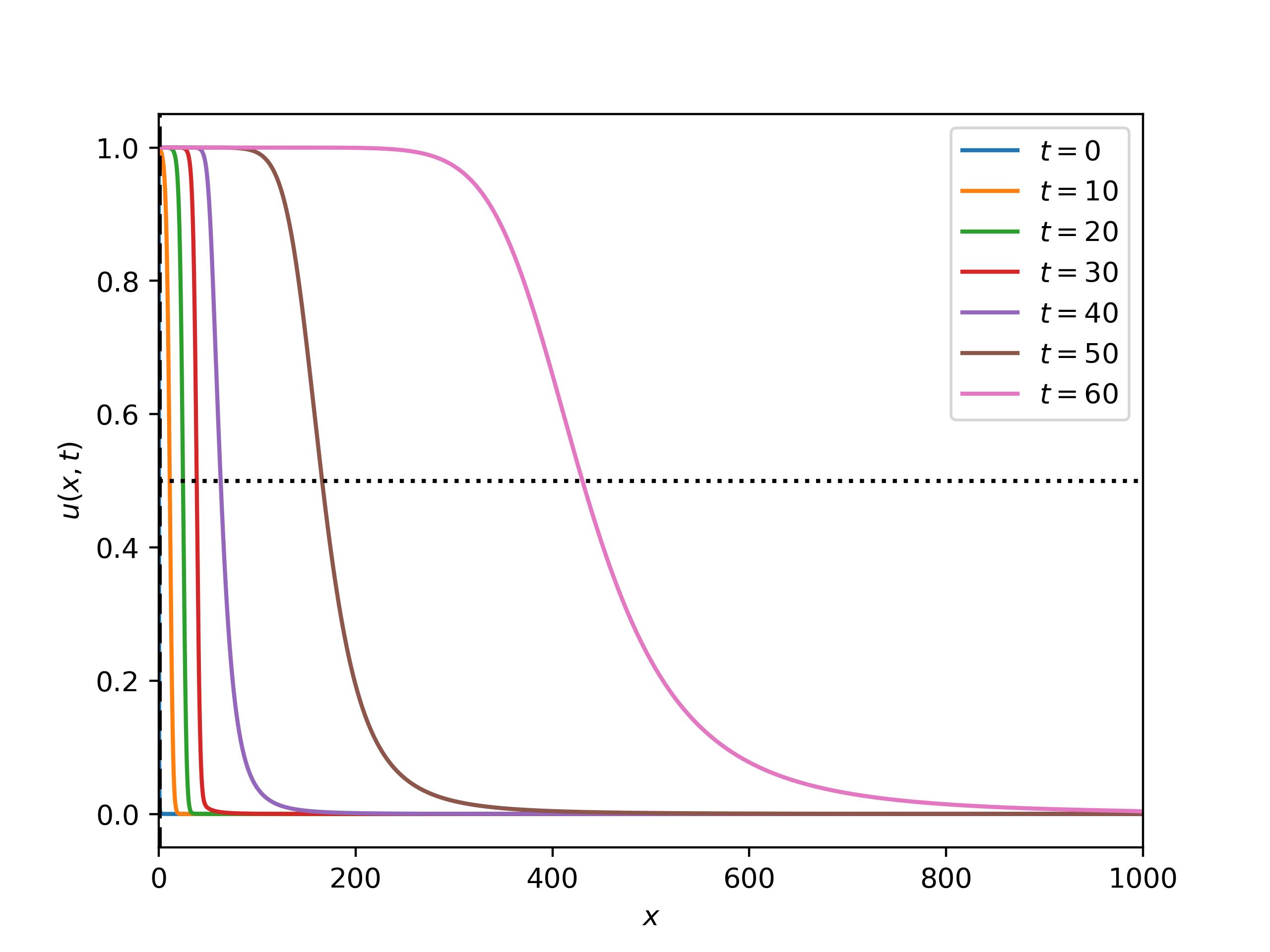} \caption{$\beta=3.0$} 
    \end{subfigure} 
\caption{Numerical approximations of the solution to \eqref{oeq1} with the initial data $u_0(x)=\min\{\frac{1}{1+100x^\beta},1\}$ at different times for $\alpha=0.4$ and different values of $\beta$.}
\label{fig:alge0.4}
\end{figure}
\begin{figure} 
\centering 
    \begin{subfigure}[b]{0.31\textwidth} 
        \includegraphics[width=\textwidth]{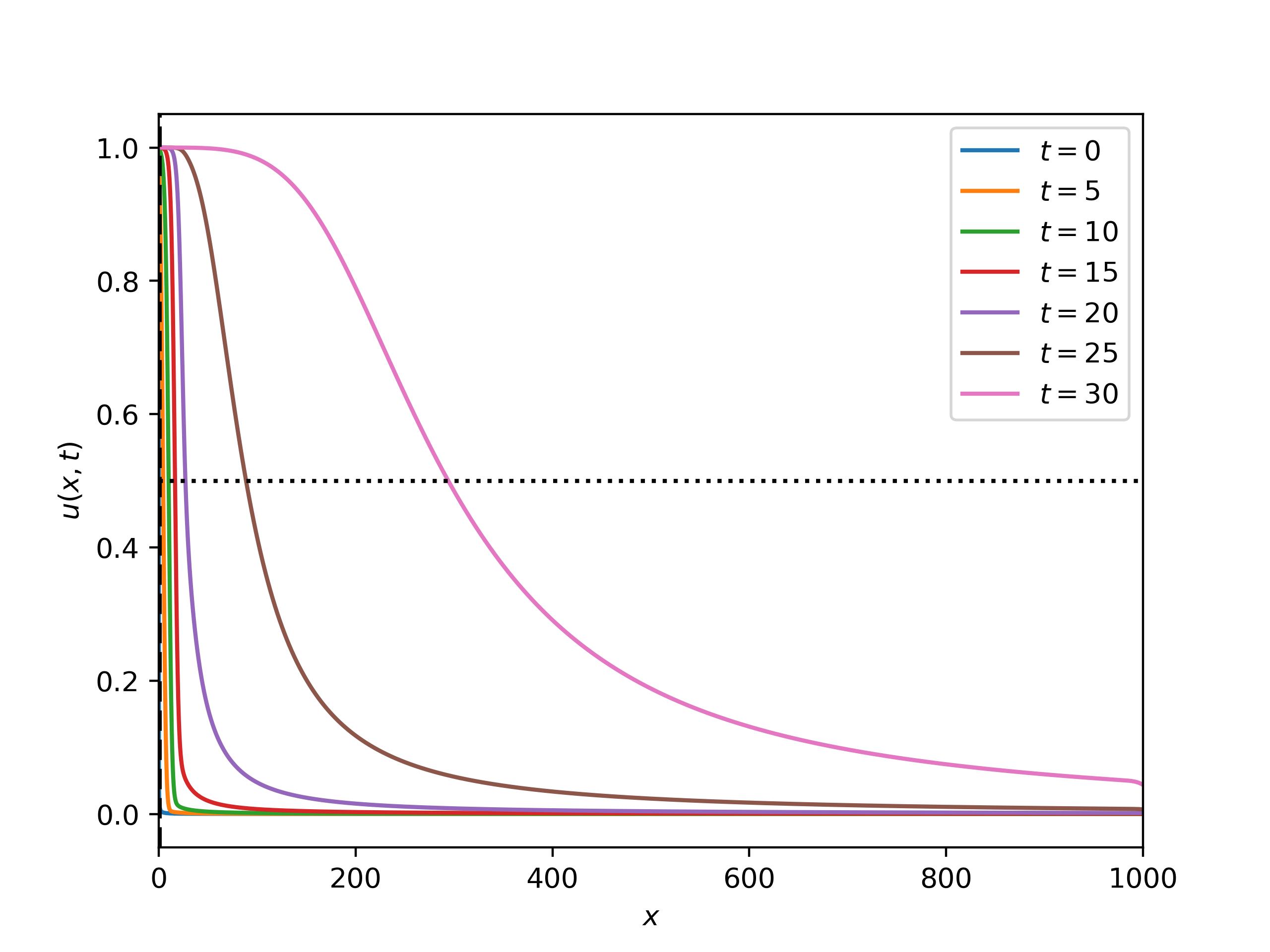} \caption{$\beta=1.0$} 
    \end{subfigure} 
    \quad 
    \begin{subfigure}[b]{0.31\textwidth}  
        \includegraphics[width=\textwidth]{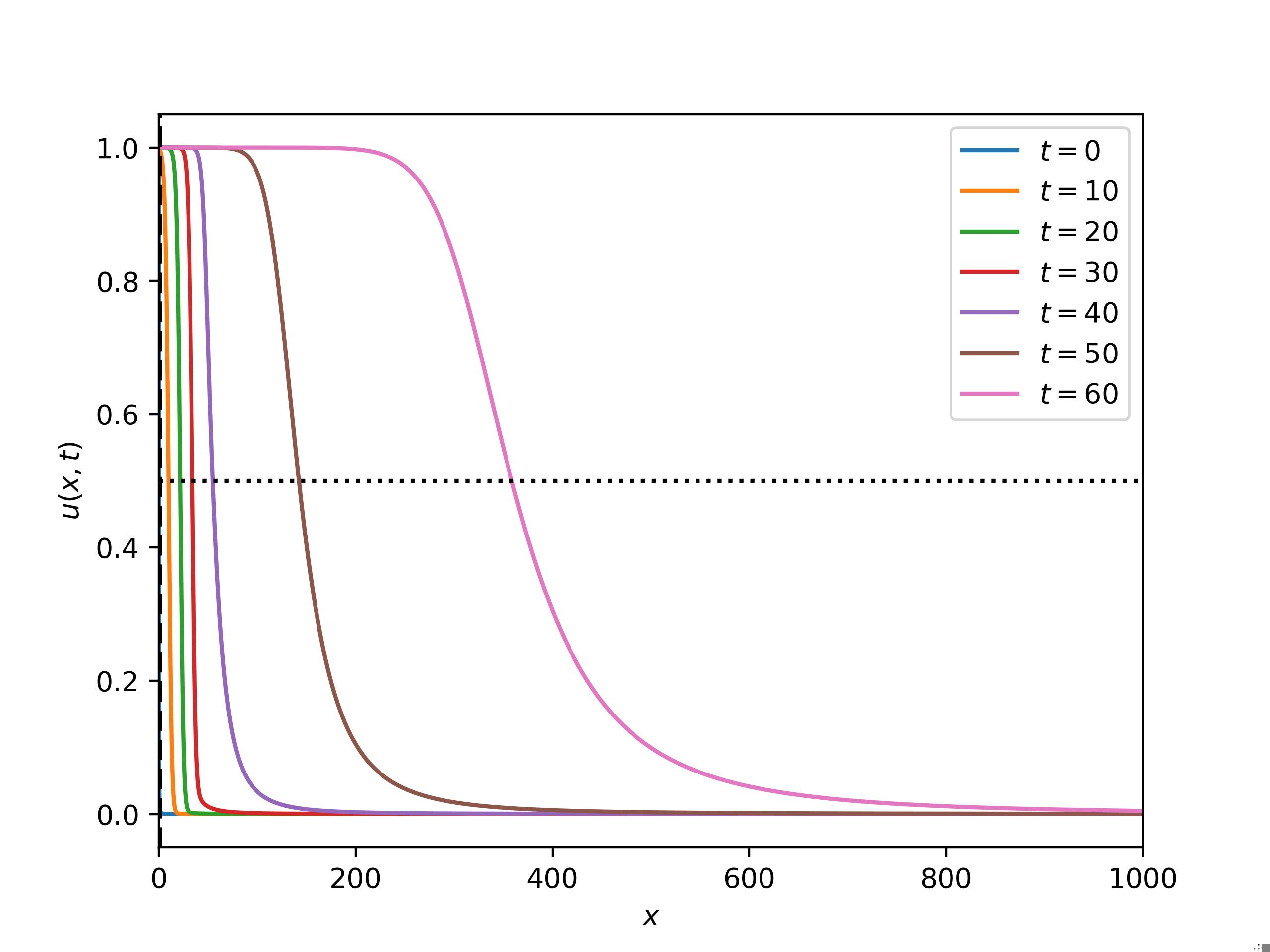} 
        \caption{$\beta=2.0$} 
    \end{subfigure} 
    \quad
    \begin{subfigure}[b]{0.31\textwidth}  
        \includegraphics[width=\textwidth]{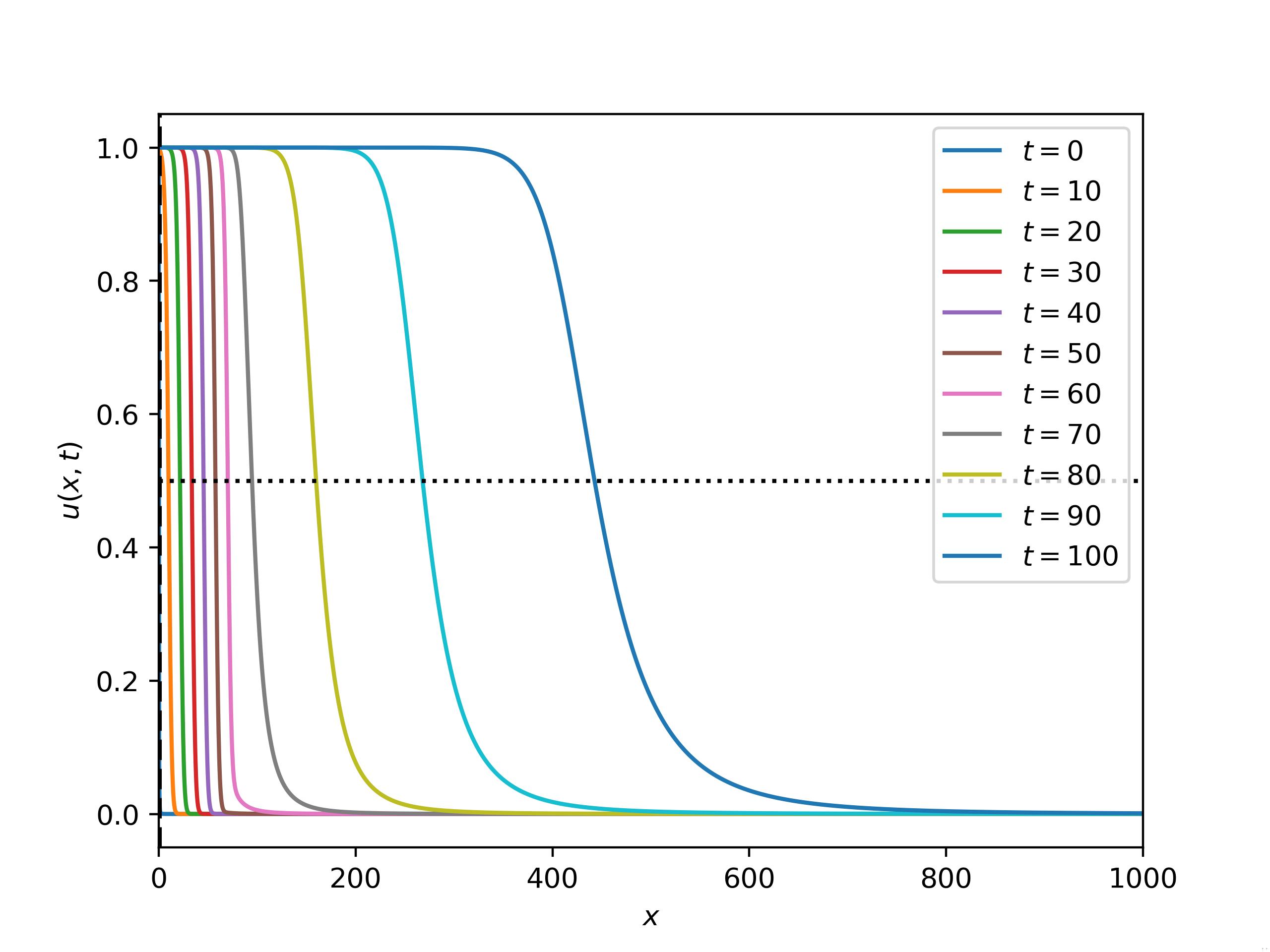} \caption{$\beta=3.0$} 
    \end{subfigure} 
\caption{Numerical approximations of the solution to \eqref{oeq1} with the initial data $u_0(x)=\min\{\frac{1}{1+100x^\beta},1\}$ at different times for $\alpha=0.6$ and different values of $\beta$.}
\label{fig:alge0.6}
\end{figure}

\begin{figure}
    \centering
    \includegraphics[width=0.65\linewidth]{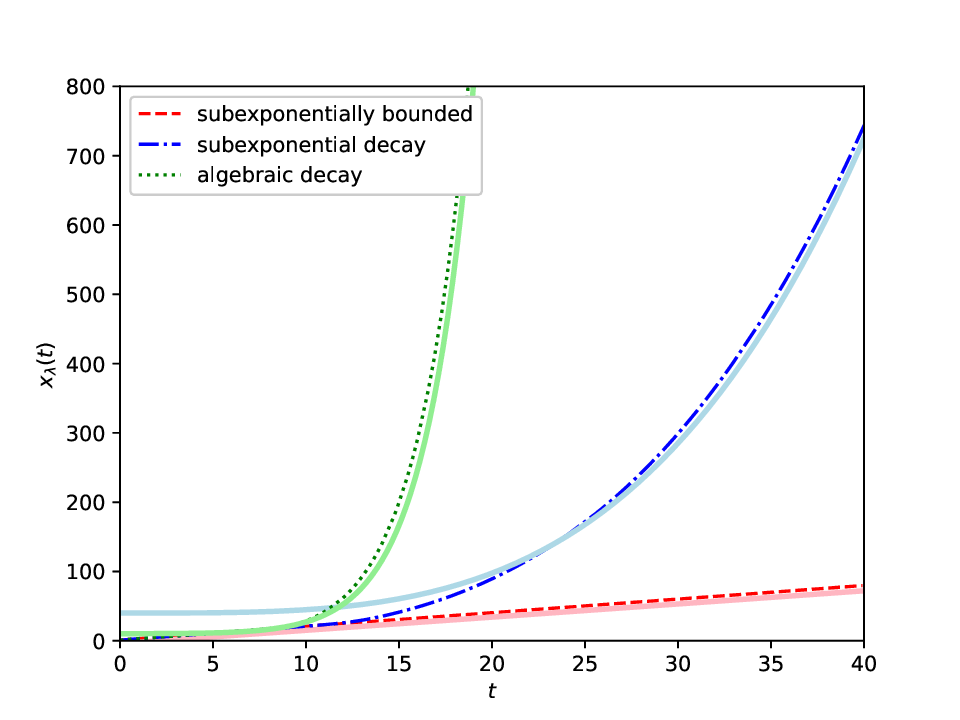}
    \caption{Comparison between the largest element $x_{\lambda}(t)$ of level sets $E_{\lambda}(t)$ of the solution starting from three types of initial data: sub-exponentially bounded $u_0(x)=\min\{e^{-5x},1\}$, sub-exponential decay $u_0(x)=\min\{e^{-5x^{0.2}},1\}$ and algebraic decay $u_0(x)=\min\{\frac{1}{1+100x},1\}$. In this figure, the thick continuous curves are theoretical results. Here, we choose $\alpha=0.4$ and $\lambda=\frac{1}{2}$.}
    \label{fig:com_level_sets}
\end{figure}

  \newpage
	\bibliographystyle{plain}
		\bibliography{ref}	
	\end{document}